\documentclass[10pt]{amsart}
\usepackage{amssymb}
\usepackage{bm}
\usepackage{graphicx}
\usepackage[centertags]{amsmath}
\usepackage{amsfonts}
\usepackage{amsthm}
\usepackage{graphicx}
\newtheorem{thm}{Theorem}
\newtheorem{cor}[thm]{Corollary}
\newtheorem{lem}[thm]{Lemma}

\newtheorem{prop}[thm]{Proposition}

\newtheorem{defn}[thm]{Definition}
\theoremstyle{definition}

\newcommand{\rr}{\mathbb{R}}
\newcommand{\nn}{\mathbb{N}}
\newcommand{\ee}{\varepsilon}
\newcommand{\ff}{\mathcal{F}}

\newcommand{\g}{\mathcal{G}}

\newcommand{\meg}{\geqslant}
\newcommand{\mik}{\leqslant}

\makeindex

\begin{document}

\title[The structure of the higher order spreading models]{On the structure of the set of higher order spreading models}

\author{B\"unyamin Sar\i\ and Konstantinos Tyros}

\address{Department of Mathematics, University of North Texas, Denton, TX 76203--5017 }
\email{bunyamin@unt.edu}

\address{Mathematics Institute, University of Warwick, Coventry, CV4 7AL, UK}
\email{k.tyros@warwick.ac.uk}
\thanks{The first author was partially supported by Simons Foundation Grant \#208290.}

\thanks{The second author was partially supported by ERC grant 306493}



\begin{abstract}
We generalize some results concerning the classical notion of a spreading model for the spreading models of order $\xi$. Among them, we prove that the set $SM_\xi^w(X)$ of the $\xi$-order spreading models of a Banach space $X$ generated by subordinated weakly null $\ff$-sequences endowed with the pre-partial order of domination is a semi-lattice.
Moreover, if $SM_\xi^w(X)$  contains an increasing sequence of length $\omega$ then it contains an increasing sequence of length $\omega_1$.
Finally, if $SM_\xi^w(X)$ is uncountable, then it contains an antichain of size the continuum.
\end{abstract}
\maketitle

\section{Introduction}
In 1974, A. Brunel and L. Sucheston introduced the notion of a spreading model, a notion that possesses a quite central role in the asymptotic Banach space theory (see for example \cite{AK,Kr,OS}). A higher order extension of this notion has been introduced and studied in \cite{AKT2} and \cite{AKT1}. In particular, for every countable ordinal $\xi$, the $\xi$-order spreading models of a Banach space $X$ are defined. The order one spreading models coincide with the classical ones. In this note, we extend some of the results concerning the structure of the set of classical spreading models to the setting of the $\xi$-order ones. Consider the set $SM_w(X)$ of (equivalence classes of) spreading models of a Banach space $X$ generated by weakly null sequences endowed with the partial order given by the domination of bases. This partially ordered set is proven to have interesting features:

(i) Every countable subset of $SM_w(X)$ admits an upper bound in the set, in particular, $SM_w(X)$ is an upper semi-lattice \cite{AOST};

(ii) Existence of an increasing sequence (of length $\omega$) in $SM_w(X)$ yields the existence of an increasing sequence of length $\omega_1$ \cite{Sa};

(iii) Suppose $X$ is separable. If $SM_w(X)$ is uncountable, then it contains an antichain of size the continuum. If $SM_w(X)$ contains a decreasing sequence of length $\omega_1$ then it contains an increasing sequence of length $\omega_1$. If $SM_w(X)$ does not contain any infinite increasing sequence then there exists $\zeta<\omega_1$ such that $SM_w(X)$ contains no decreasing sequence of length $\zeta$ \cite{Do}.

We show that for every $\xi<\omega_1$, these results extend to the set $SM_\xi^{w}(X)$ of $\xi$-order spreading models of $X$ generated by subordinated weakly null $\ff$-sequences. Subordinated $\ff$-sequences are higher order analogue of the ordinary weakly convergent sequences.

A brief background to the higher order spreading models and some new facts regarding subordinated $\ff$-sequences are given in Section  \ref{section_prel} and \ref{section_subord}. Each of the following Sections \ref{section_AOST}, \ref{section_S} and \ref{section_D} is devoted to the proof of the generalization of the results (i), (ii), and (iii) mentioned above, respectively.  In particular, the main results of the paper are Corollary \ref{long_chains} and Theorems \ref{lattice}, \ref{lattice_countable}, \ref{thm_gen_D} and \ref{thm_gen_D_2}.

It is worth to mention that in general $SM_\xi^{w}(X)$ does not coincide with $SM_1^{w}(X)$, and therefore the transfinite hierarchy $(SM_\xi^{w}(X))_{\xi<\omega_1}$ is not trivial. In fact, for every $k$ there exists a Banach space $X$ such that $SM_k^{w}(X)$ is a proper subset of $SM_{k+1}^{w}(X)$ \cite{AKT1}. Moreover, there are reflexive Banach spaces $X$ and $Y$ which have, up to equivalence, the same set of spreading models of the first order but not of the second order. In Section \ref{section_subord} we also recall these known examples.

\section{The $\xi$-order spreading models of a Banach space}\label{section_prel}
By
$\nn=\{1,2,\ldots\}$ we denote the set of all positive integers.
We will use capital letters $L,M,N,...$
to denote infinite subsets and lower case letters $s,t,u,...$ to
denote finite subsets of $\nn$. For every infinite subset $L$ of
$\nn$, $[L]^{<\infty}$ (resp. $[L]^\infty$) stands for the set of
all finite (resp. infinite) subsets of $L$. For an
infinite subset $L=\{l_1<l_2<...\}$ of $\nn$ and a positive integer
$k\in\nn$, we set $L(k)=l_k$. Similarly, for a finite subset
$s=\{n_1<..<n_m\}$ of $\nn$ and  for $1\mik k\mik m$ we set
$s(k)=n_k$. For an infinite subset $L=\{l_1<l_2<...\}$ of $\nn$ and a finite
subset $s=\{n_1<..<n_m\}$ (resp. for an infinite subset
$N=\{n_1<n_2<...\}$ of $\nn$), we set
$L(s)=\{l_{n_1},...,l_{n_m}\}=\{L(s(1)),...,L(s(m))\}$ (resp.
$L(N)=\{l_{n_1},l_{n_2},...\}=\{L(N(1)), L(N(2)),...\}$).


The $\ff$-spreading models are generated by sequences indexed by the elements of
a regular thin family, which we are about to define. Let
$\mathcal{R}$ be a family of finite subsets of $\nn$. The family $\mathcal{R}$
is called \emph{compact} if the set of all characteristic functions of the
elements of $\mathcal{R}$ forms a compact subset of the set of all functions
from $\nn$ into $\{0,1\}$ endowed with the product topology. The family
$\mathcal{R}$ is called \emph{hereditary} if it contains every subset of
its elements. The family $\mathcal{R}$ is called \emph{spreading} if for
every $s\in \mathcal{R}$ and $t\in[\nn]^{<\infty}$ of the same cardinality,
say $n$, such that $s(i)\mik t(i)$ for all $1\mik i\mik n$, we have that $t$
also belongs to $\mathcal{R}$.
A family of finite subsets of $\nn$ is called \emph{regular} if it is
compact, hereditary and spreading. Finally, a family
$\ff$ of finite subsets of $\nn$ is called \emph{regular thin}, if it consists of the
maximal elements, under inclusion, of some regular family $\mathcal{R}$. A family
$\mathcal{H}$ of finite subsets of $\nn$ is called \emph{thin} if there are no
$s$ and $t$ in $\mathcal{H}$ such that $s$ is proper initial segment of $t$.
Clearly every regular thin family is also thin.
A brief presentation of the regular and the regular thin families as well as the relation
between them can be found in Section 2 of
\cite{AKT2}.
Finally, given a regular thin family $\ff$,
an $\ff$-sequence in a Banach space is a sequence of the form
$(x_s)_{s\in\ff}$ indexed by $\ff$, while an $\ff$-subsequence is a sequence
of the form $(x_s)_{s\in\ff\upharpoonright L}$ indexed by $\ff\upharpoonright L$,
where $L$ is an infinite subset of $\nn$ and the restriction
$\ff\upharpoonright L$ of $\ff$ on $L$ is defined by
\begin{equation}
\label{eq001}
\ff\upharpoonright L =\{s\in\ff:s\subseteq L\}.
\end{equation}
The connection between the $\ff$-spreading models and the $\ff$-sequences generating
them is described by the notion of plegma families.
\begin{defn}\label{defn plegma}
    Let $l$ be a positive integer and $s_1,...,s_l$ nonempty finite subsets of $\nn$.
    The $l-$tuple $(s_j)_{j=1}^l$ is called a \textit{plegma}
  family if the following are satisfied.
    \begin{enumerate}
      \item[(i)] For every $i,j$ in $\{1,...,l\}$ and $k$ in $\nn$ with $i<j$ and $k\mik\min(|s_i|,|s_j|)$, we have that $s_i(k)<s_j(k)$.
      \item[(ii)] For every $i,j$ in $\{1,...,l\}$ and $k$ in $\nn$ with $k\mik \min (|s_i|,|s_j| -1)$, we have that $s_i(k)<s_j(k+1)$.
    \end{enumerate}
  \end{defn}

  For instance,
a pair of doubletons $(\{n_1,m_1\},\{n_2,m_2\})$ is plegma if and only if $n_1<n_2<m_1<m_2$.
More generally for two non empty $s,t\in [\nn]^{<\infty}$ with
$|s|\mik |t|$ the pair   $(s,t)$ is a plegma pair if and only if
$s(1)<t(1)<s(2)<t(2)<...<s(|s|)<t(|s|)$.

The properties of the plegma families are explored in Section 3 of \cite{AKT2}.
Moreover, for every regular thin family $\ff$, every infinite
subset $L$ of $\nn$ and every positive integer $k$ we set
\begin{equation}
  \label{eq002}
  \mathrm{Plm}_k(\ff\upharpoonright L)=\big\{(s_i)_{i=1}^k:s_1,...,s_k\in \ff\upharpoonright L\text{
  such that }(s_i)_{i=1}^k\text{ plegma}\big\}.
\end{equation}
Now we are ready to state the definition of the $\ff$-spreading models.
\begin{defn}\label{Definition_of_spreading_model}
   Let   $X$ be a Banach space, $\ff$ a regular thin family and
   $(x_s)_{s\in\ff}$ an $\ff$-sequence in $X$. Also let $(E,\|\cdot\|_*)$ be an infinite dimensional
   seminormed linear space with Hamel basis
   $(e_n)_{n}$. Finally, let $M$ in $[\nn]^\infty$ and $(\delta_n)_n$
   be a null sequence of positive real numbers.

  We say that the $\ff$-subsequence
  $(x_s)_{s\in\ff\upharpoonright M}$ generates  $(e_n)_{n}$ as an
  $\ff$-spreading model \big(with respect to $(\delta_n)_n$\big) if for every $l,k$ in $\nn$ with $1\mik k\mik l$, every finite sequence $(a_i)_{i=1}^k$  in $[-1,1]$ and every
$(s_j)_{j=1}^k$ in $\mathrm{{Plm}}_k(\ff\upharpoonright M)$
  with $s_1(1)\geq M(l)$, we have
  \begin{equation}
  \label{eq003}
  \Bigg{|}\Big{\|}\sum_{j=1}^k a_j x_{s_j}\Big{\|}-\Big{\|}\sum_{j=1}^k a_j
  e_j\Big{\|}_* \Bigg{|}\mik\delta_l.
  \end{equation}
We also say that $(x_s)_{s\in\ff}$ admits $(e_n)_{n}$ as an
$\ff$-spreading model
   if there exists  $M$ in $[\nn]^\infty$ such that  $(x_s)_{s\in\ff\upharpoonright M}$ generates  $(e_n)_{n}$ as an
  $\ff$-spreading model.

  Finally, for a subset $A$ of $X$,  $(e_n)_{n}$ is an $\ff$-spreading model of $A$ if there exists an $\ff$-sequence $(x_s)_{s\in\ff}$ in $A$ which admits $(e_n)_{n}$ as an $\ff$-spreading model.
\end{defn}

The existence of the $\ff$-spreading models is established by Theorem 4.5 in \cite{AKT2}.
For every regular family $\mathcal{R}$ its \emph{order} $o(\mathcal{R})$ is defined to be the rank
of the partially ordered set $\mathcal{R}$ endowed with the reverse inclusion.
By the compactness of $\mathcal{R}$ its order is well defined and it is a countable
ordinal number (see also Section 2 of \cite{AKT2}). The \emph{order} $o(\ff)$ of a regular thin family
$\ff$ is defined to be the order of the regular family for which $\ff$ is the set of maximal elements under inclusion. It turns out that the order of the regular thin family $\ff$ involved
in the definition of the $\ff$-spreading models is the important feature in the following sense.
By Corollary 4.7 from \cite{AKT2}, for every subset $A$ of a Banach space
and every pair $\ff,\g$ of regular thin families of the same order, a
sequence $(e_n)_n$ is an $\ff$-spreading model of $A$ if and only if it is a
$\g$-spreading model of $A$. This fact gives rise to the following definition.
\begin{defn}
  Let $A$ be a subset of a Banach space $X$ and $\xi\meg 1$ be
  a countable ordinal. We say that $(e_n)_{n}$ is a $\xi$-order spreading model of $A$ if there
  exists a regular thin family $\ff$ with $o(\ff)=\xi$ such that
  $(e_n)_{n}$ is an $\ff$-spreading model of $A$.
\end{defn}

An $\ff$-subsequence $(x_s)_{s\in\ff\upharpoonright L}$ in a Banach
space $X$ is defined to be \emph{weakly null} if for every $x^*$ in the dual of $X$
and every $\ee>0$ there exists some $n_0$ such that for every
$s\in\ff\upharpoonright  L$ with $\min s\meg n_0$, we have that $|x^*(x_s)|<\ee$.
For a regular thin family $\ff$ its closure
is defined by
\begin{equation}
\label{eq004}
\widehat{\ff}=\{s\in[\nn]^{<\infty}:\text{ there exists }t\in \ff\text{ such that }
s\subseteq t\}.
\end{equation}
Identifying each subset of $\nn$ with its characteristic map and
endowing $\{0,1\}^\nn$ with the product topology, we have that $\ff$
(resp. $\ff\upharpoonright L$) is a discrete subset of $\{0,1\}^\nn$ and
$\widehat{\ff}$ (resp. $\widehat{\ff}\upharpoonright L$) is its
topological closure.
\begin{defn}
Let $\ff$ be a regular thin family and $M$ an infinite subset of $\nn$. Let $(x_s)_{s\in\ff}$ an $\ff$-sequence in a Banach space $X$.
We say that the $\ff$-subsequence $(x_s)_{s\in\ff\upharpoonright L}$ is
subordinated (with respect to the weak topology)
if there exists a continuous map $\widehat{\varphi}:\widehat{\ff}\upharpoonright L\to X$,
where $X$ is considered with the weak topology, such that $x_s=\widehat{\varphi}(s)$ for all
$s\in \ff\upharpoonright L$.
\end{defn}
It is easy to see that
a subordinated $\ff$-subsequence $(x_s)_{s\in\ff\upharpoonright L}$
is weakly convergent to $\widehat{\varphi}(\emptyset)$, where $\widehat{\varphi}$ is
the map witnessing the fact that $(x_s)_{s\in\ff\upharpoonright L}$ is subordinated.
Thus a subordinated $\ff$-sequence $(x_s)_{s\in\ff\upharpoonright L}$ is weakly null if and only
if $\widehat{\varphi}(\emptyset)=0$.
\begin{defn}
Let $X$ be a Banach space and $\xi$ a countable ordinal. By
$SM_\xi^w(X)$ we denote the set of all $\xi$-order spreading models
generated by subordinated weakly null $\ff$-subsequences in $X$,
for some $\ff$ regular thin family of order $\xi$.
\end{defn}

\section{On subordinated weakly null $\ff$-sequences}\label{section_subord}
Our first goal is the following strengthening of Theorem 3.17
from \cite{AKT2}.
\begin{thm}\label{continuous_plegma_embending}
  Let $\ff$ and $\g$ be regular thin families with $o(\ff)\mik o(\g)$. Then for every
  infinite subsets $M$ and $N$ of $\nn$ there exist $M'$ in $[M]^\infty$,
  $N'$ in $[N]^\infty$ and a continuous map
  $\widehat{\varphi}:\widehat{\g}\upharpoonright N'\to \widehat{\ff}\upharpoonright M'$ satisfying the following. Let $\varphi$ be the restriction of $\widehat{\varphi}$
  on $\g \upharpoonright N'$. Then $\varphi$ is a plegma preserving map
  onto $\ff\upharpoonright M'$ such that $\min\varphi(t)\meg M'(l)$
  for every $l\in\nn$ and
  $t\in \g \upharpoonright N'$ with $\min t\meg N'(l)$.
\end{thm}
For the proof of the above theorem we will need Corollary 2.17 from
\cite{AKT2}.
To state it we recall some notation.
For two families $\mathcal{H}_1$ and $\mathcal{H}_2$ of finite subsets of $\nn$, write
$\mathcal{H}_1\sqsubseteq\mathcal{H}_2$ if every element in $\mathcal{H}_1$ has
an extension in $\mathcal{H}_2$ and
every element in $\mathcal{H}_2$ has an  initial segment in $\mathcal{H}_1$.
Moreover, for every infinite subset $L$ of $\nn$ and every family $\mathcal{H}$ of
finite subsets of $\nn$, let
\begin{equation}
  \label{eq005}
  L(\mathcal{H})=\{L(s):s\in\mathcal{H}\}.
\end{equation}

\begin{prop}[Corollary 2.17 \cite{AKT2}]
 \label{cx21}
 Let $\ff$ and $\g$ be regular thin families with $o(\ff)\mik o(\g)$.
 Then there exists $L_0$ in $[\nn]^\infty$ such that for every
 $M$ in $[\nn]^\infty$ there exists $L$ in $[L_0(M)]^\infty$ satisfying $L_0(\ff)\upharpoonright L\sqsubseteq\g\upharpoonright L$.
\end{prop}

\begin{proof}[Proof of Theorem \ref{continuous_plegma_embending}]
Let $M$ and $N$ be two infinite subsets of $\nn$. For every finite subset $t$
of $\nn$, there exists a unique finite subset $s$ of $\nn$ such that $N(s)=t$.
Put $N^{-1}(t)=s$, and
$\g'=\{N^{-1}(t)\in[\nn]^{<\infty}:t\in \g\}$. It follows that $\g'$ is
a regular thin family with $o(\g')=o(\g)$. Moreover,
\begin{equation}
  \label{eq006}
  N(\g')=\g\upharpoonright N.
\end{equation}
By Proposition \ref{cx21}, there exist infinite subsets $L_0$ of $\nn$
and $L$
of $L_0(M)$ such that
\begin{equation}
  \label{eq007}
  L_0(\ff)\upharpoonright L\sqsubseteq\g'\upharpoonright L.
\end{equation}
We essentially need to prove the
following.

  \noindent \textbf{Claim:} There exists a continuous map
  $\widehat{\varphi}_1:\widehat{\g}'\upharpoonright L\to L_0(\widehat{\ff})\upharpoonright L$
  such that, setting $\varphi_1$ to be the restriction of $\widehat{\varphi}_1$
  on $\g'\upharpoonright L$, the following is satisfied. The map $\varphi_1$ is a plegma preserving map
  onto $L_0(\ff)\upharpoonright L$ such that for every
  $t\in \g' \upharpoonright L$,  $\varphi_1(t)$
  is an initial segment of $t$ and therefore $\min t=\min\varphi_1(t)$.

  \noindent{\em Proof of Claim.} Since $\ff$ is a thin family,
     so is $L_0(\ff)$.
     Hence, by \eqref{eq007}, for every $t\in\g'\upharpoonright L$ there
     exists a unique $s_t\in L_0(\ff)$ such that $s_t$ is an initial segment of $t$. Let
     $\mathcal{A}= (\widehat{\g}'\upharpoonright L)\setminus
     (L_0(\widehat{\ff})\upharpoonright L)$.
     Also observe that for every $\widehat{t}\in\mathcal{A}$ and every $t,t'$ in
     $\g'\upharpoonright L$ that both end-extend $\widehat{t}$ we have that
     $s_t=s_{t'}$ and both $s_t,s_{t'}$ are initial segments of $\widehat{t}$. For every $\widehat{t}\in\mathcal{A}$ we set
     $s_{\widehat{t}}=s_t$ where $t$ is any element from $\g'\upharpoonright L$ that
     end-extends $\widehat{t}$. Finally, for every $\widehat{t}\in L_0(\widehat{\ff})\upharpoonright L$
     we set $s_{\widehat{t}}=\widehat{t}$. Setting $\widehat{\varphi}_1(\widehat{t})=s_{\widehat{t}}$ for all
     $\widehat{t}\in\g'\upharpoonright L$, we have that $\widehat{\varphi}_1$ is
     as desired and the proof of the claim is complete.

  Since $L\in[L_0(M)]^\infty$, there exists $M'$ in $[M]^\infty$ such that $L_0(M')=L$. Then $L_0(\widehat{\ff}\upharpoonright M')=L_0(\widehat{\ff})\upharpoonright L$.
  Moreover, let $N'=N(L)$. It is easy to check that $\widehat{\g}\upharpoonright N'=N(\widehat{\g}'\upharpoonright L)$.
  Define $\widehat{\varphi}:\widehat{\g}\upharpoonright N'\to \widehat{\ff}\upharpoonright M'$
  by setting for every $\widehat{t}\in \widehat{\g}\upharpoonright N'$,
  $\widehat{\varphi}(\widehat{t})=L_0^{-1}(\widehat{\varphi}_1(N^{-1}(\widehat{t})))$, where $L_0^{-1}(s)$
  is defined similarly to $N^{-1}(s)$ for every $s\in[L_0]^{<\infty}$. It follows readily
  that $\widehat{\varphi}$ is as desired and the proof is complete.
\end{proof}

By Theorem \ref{continuous_plegma_embending}, we have the following
immediate corollary.
\begin{cor}
\label{generating_subord_sm}
Let $X$ be a Banach space, $\xi$ a countable ordinal and $(e_n)_n\in SM_\xi^w(X)$. Then for every regular
thin family $\ff$ of order at least $\xi$ and every infinite subset $M$ of $\nn$ there exist a
further infinite subset $L$ of $M$  and an $\ff$-sequence $(x_s)_{s\in\ff}$
in $X$ such that the $\ff$-subsequence $(x_s)_{s\in\ff\upharpoonright L}$ is
subordinated, weakly null and generates $(e_n)_n$ as an $\ff$-spreading model.
\end{cor}
\begin{proof}
Since $(e_n)_n$ belongs to $SM_\xi^w(X)$, there exist a regular thin family
$\g$ of order $\xi$, an infinite subset $N$ of $\nn$ and a
$\g$-sequence $(y_s)_{s\in\g}$ such that the $\g$-subsequence
$(y_s)_{s\in\g\upharpoonright N}$ is subordinated, weakly null and generates
$(e_n)_n$ as a $\g$-spreading model. Let $\widehat{\varphi}_1:
\widehat{\g}\upharpoonright N\to X$ be the continuous map witnessing that
the $\g$-subsequence
$(y_s)_{s\in \g\upharpoonright N}$ is subordinated.
Fix a regular thin family $\ff$ of order at least $\xi$ and
an infinite subset $M$ of $\nn$.
By Theorem \ref{continuous_plegma_embending}
 there exist an infinite subset $L$ of $M$, an infinite subset $N'$ of $N$
and a continuous map $\widehat{\varphi}:\widehat{\ff}\upharpoonright L\to
\widehat{\g}\upharpoonright N'$ such that setting $\varphi$ to be the
restriction of $\widehat{\varphi}$ on $\ff \upharpoonright L$ we have
that $\varphi$ is a plegma preserving map onto $\g\upharpoonright N'$
satisfying for every $l\in\nn$ and $t\in \ff \upharpoonright L$ ,
\begin{equation}
  \label{eq008}
  \text{if }\min t\meg L(l)\text{ then }\min\varphi(t)\meg N'(l).
\end{equation}
Set
$\widehat{\varphi}_2=\widehat{\varphi}_1\circ\widehat{\varphi}$ and $x_s=y_{\widehat{\varphi}_2(s)}$
for all $s\in\widehat{\ff}\upharpoonright L$.
Then $\widehat{\varphi}_2$ is continuous and therefore $(x_s)_{s\in\ff\upharpoonright L}$ is subordinated.
Since the $\g$-subsequence $(y_s)_{s\in\g\upharpoonright N}$ generates $(e_n)_n$ as a $\g$-spreading model,
it follows that $(y_s)_{s\in\g\upharpoonright N'}$ also does and therefore, invoking \eqref{eq008},
we get that $(x_s)_{s\in\ff\upharpoonright L}$ generates $(e_n)_n$ as an $\ff$-spreading model too.
The continuity of $\widehat{\varphi}$ and \eqref{eq008} yields that
$\widehat{\varphi}(\emptyset)=\emptyset$. Thus, since $(y_s)_{s\in\g\upharpoonright N}$ is weakly
null, $(x_s)_{s\in\ff\upharpoonright L}$ is weakly null too.
\end{proof}

Corollary \ref{generating_subord_sm} yields that the transfinite hierarchy $(SM_\xi^w(X))_{\xi<\omega_1}$ is an increasing one. In general, this hierarchy is nontrivial and the class $SM_{\xi}^w(X)$ for $\xi>1$ is richer than $SM_1^w(X)$.  In fact, for every positive integer $k$ there exists a reflexive space $\mathfrak{X}_{k+1}$ with an unconditional basis such that $\mathfrak{X}_{k+1}$ has no $k$-order spreading model equivalent to the standard basis of $\ell^1$ while $\mathfrak{X}_{k+1}$ admits a
 $k+1$-order spreading model equivalent to the standard basis of $\ell^1$ (Section 12, \cite{AKT1}). Since $\mathfrak{X}_{k+1}$ is reflexive, it follows that $SM_{k}^w(\mathfrak{X}_{k+1})$ is a proper subset of $SM_{k+1}^w(\mathfrak{X}_{k+1})$. This is due to the fact that for reflexive Banach spaces $X$ the set of (Schauder basic) spreading models of order $\xi$ coincides with the set of spreading models generated by subordinated weakly null $\ff$-sequences of the same order. Indeed, if $(e_n)$ is a spreading model of order $\xi$ and $\ff$ a regular thin family of order $\xi$, then by the reflexivity of $X$ and Proposition 6.16 of \cite{AKT2}, $(e_n)$ is generated by a subordinated $\ff$-subsequence $(x_s)_{s\in\ff\upharpoonright M}$. If $(e_n)$ is not equivalent to $\ell^1$ basis, then by  Theorem 6.14 of \cite{AKT2} $(x_s)_{s\in\ff\upharpoonright M}$ is weakly null. If $(e_n)$ is equivalent to $\ell^1$ basis, then, again by  Theorem 6.14 of \cite{AKT2}, $(e_n)$ belongs to $SM_\xi^w(X)$.

It is important to point out that the higher order spreading models generated by subordinated weakly null $\ff$-sequences do form a new isomorphic invariance for Banach spaces. In particular, there exist two Banach spaces $X$ and $Y$ such that $SM_1^w(X)=SM_1^w(Y)$ and $SM_2^w(X)\neq SM_2^w(Y)$. This was known to the authors of \cite{AKT3}, though it wasn't explicitly stated.  Indeed, let $X$ be the space $X^2_{T,2,3}$ given in Theorem 12.11 in \cite{AKT3} and $Y$ the direct sum of the Tsirelson space and $\ell^2$. Both spaces are reflexive and therefore for every countable ordinal $\xi$ and every Schauder basic $\xi$-order spreading model $(e_n)_n$ of $X$ (resp. $Y$), we have that $(e_n)_n$ belongs to $SM_\xi^w(X)$ (resp. $SM_\xi^w(Y)$). By Theorems 7.3 and 9.18 in \cite{AKT3}, it is easy to see that every Schauder basic spreading model of $Y$ of any order $\xi$ is equivalent to either the standard basis of $\ell^1$ or the standard basis of $\ell^2$, and, of course, both cases occur for every $\xi$. On the other hand, it is shown there that while every $(e_n)_{n}$ in $SM_1^w(X)$ is equivalent to either the standard basis of $\ell^1$ or the standard basis of $\ell^2$, $SM_2^w(X)$ contains a sequence equivalent to the standard basis of $\ell^3$.

\section{The semi-lattice structure of $SM_\xi^w(X)$}\label{section_AOST}
Let $X$ be a Banach space and $\xi$ be a countable ordinal.
By Theorem 6.11 of \cite{AKT2}, every sequence $(e_n)_n$ in $SM_\xi^w(X)$ is either $1$-suppression unconditional or $\|\sum_{i=1}^na_ie_i\|=0$ for all $n\in \nn$ and $(a_i)_1^n\in\rr$.
We endow $SM_\xi^w(X)$ with the pre-partial order $\preccurlyeq$ of
domination. That is, for two sequences $(e^1_n)_n$ and $(e^2_n)_n$ in
$SM_\xi^w(X)$ and $C>0$ we say that $(e^2_n)_n$ $C$-dominates $(e^1_n)_n$ if
\begin{equation}
  \label{eq009}
  \Big\|\sum_{i=1}^na_ie^1_i\Big\|\mik C\Big\|\sum_{i=1}^na_ie^2_i\Big\|
\end{equation}
for all $n\in \nn$ and $(a_i)_1^n\in\rr$. Write $(e^1_n)_n\preccurlyeq(e^2_n)_n$ if $(e^2_n)_n$ $C$-dominates $(e^1_n)_n$ for some $C>0$.
$(e^1_n)_n$ and $(e^2_n)_n$ are equivalent, denoted by $(e^1_n)_n\sim(e^2_n)_n$, if $(e^1_n)_n\preccurlyeq(e^2_n)_n$ and $(e^2_n)_n\preccurlyeq(e^1_n)_n$. We write $(e^1_n)_n\prec(e^2_n)_n$, if $(e^1_n)_n\preccurlyeq(e^2_n)_n$ and $(e^1_n)_n\not\sim(e^2_n)_n$.
It is easy to see that $\sim$ is an equivalence relation on $SM_\xi^w(X)$.
Let $\mathbf{SM}_\xi^w(X)=SM_\xi^w(X)/_\sim$ endowed with the partial order induced by $\preccurlyeq$, which will be denoted again by $\preccurlyeq$.

In this section we will generalize some results from \cite{AOST}.
The arguments are similar to the ones contained in \cite{AOST}. In particular,
we will show that $\mathbf{SM}_\xi^w$ is a semi-lattice and that
every countable subset of $\mathbf{SM}_\xi^w$ admits an upper bound in $\mathbf{SM}_\xi^w$.
Towards achieving that we introduce a further generalization of the $\ff$-spreading models which we call the joint $\ff$-models.
First, we need some additional notation. For every
$k$ in $\nn$, define a map $i_k:\nn\to\nn$ by setting for every $j\in\nn$
\begin{equation}
  \label{eq010}
  i_k(j)=((j-1) \mod k) +1.
\end{equation}

\begin{defn}\label{joint_spreading_model}
   Let $X$ be a Banach space, $\ff$ a regular thin family, $k\in \nn$ and
   $((x^i_s)_{s\in\ff})_{i=1}^k$ a $k$-tuple of $\ff$-sequences in $X$.
   Let $(E,\|\cdot\|_*)$ be an infinite dimensional
   seminormed linear space with Hamel basis
   $(e_n)_{n}$. Let $M\in[\nn]^\infty$ and $\delta_n\searrow 0$

  We say that the $k$-tuple $((x^i_s)_{s\in\ff\upharpoonright M})_{i=1}^k$
  generates $(e_n)_{n}$ as a joint
  $\ff$-model \big(with respect to $(\delta_n)_n$\big) if for every $m,n$ in $\nn$ with $1\mik m\mik n$, every finite sequence $(a_i)_{i=1}^m$  in $[-1,1]$ and every
$(s_j)_{j=1}^m$ in $\mathrm{{Plm}}_m(\ff\upharpoonright M)$
  with $s_1(1)\geq M(n)$, we have
  \begin{equation}
  \label{eq011}
  \Bigg{|}\Big{\|}\sum_{j=1}^m a_j x^{i_k(j)}_{s_j}\Big{\|}-\Big{\|}\sum_{j=1}^m a_j
  e_j\Big{\|}_* \Bigg{|}\mik\delta_n.
  \end{equation}
A $k$-tuple $((x^i_s)_{s\in\ff})_{i=1}^k$ is said to admit $(e_n)_{n}$ as a joint
$\ff$-model
   if there exists  $M\in[\nn]^\infty$ such that the $k$-tuple $((x^i_s)_{s\in\ff\upharpoonright M})_{i=1}^k$ generates  $(e_n)_{n}$ as a joint
  $\ff$-model.
%
\end{defn}

Note that a joint $\ff$-model is not necessarily spreading.
The arguments establishing the existence
of the joint $\ff$-models (see Theorem \ref{existence_joint_sm} below)
are similar to the ones concerning the $\ff$-spreading models.
We will need the following result
(see Theorem 3.6 of \cite{AKT2}) which establishes
the Ramsey property for the plegma families.

\begin{thm} \label{ramseyforplegma}
Let $\ff$ be a
regular thin family, $M$ an infinite subset of $\nn$ and $l\in\nn$. Then for every finite partition
$\mathrm{{Plm}}_l(\ff\upharpoonright M)=\bigcup_{i=1}^p \mathcal{P}_i$,
there exist $L\in[M]^\infty$ and $1\mik i_0\mik p$ such that
$\mathrm{{Plm}}_l(\ff\upharpoonright L)\subseteq \mathcal{P}_{i_0}$.
\end{thm}


\begin{lem}
  Let $X$ be a Banach space and $\ff$ a regular thin family. Also let $k\in\nn$
  and  $(x_s^1)_{s\in\ff},...,(x_s^k)_{s\in\ff}$ bounded $\ff$-sequences in
  $X$. Then for every infinite subset $M$ of $\nn$, every $\ee>0$ and every $l\in\nn$,
  there exists an infinite subset $L$ of $M$ such that for every
  $(s_j)_{j=1}^l,(t_j)_{j=1}^l$ in $\mathrm{{Plm}}_l(\ff\upharpoonright L)$
  and every choice of reals
  $(a_j)_{j=1}^l$ from $[-1,1]$ we have that
  \begin{equation}
    \label{eq012}
    \Bigg|\Big\|\sum_{j=1}^l a_jx^{i_k(j)}_{s_j}\Big\|-\Big\|\sum_{j=1}^l a_jx^{i_k(j)}_{t_j}\Big\|\Bigg|<\ee,
  \end{equation}
  where $i_k$ is as defined in \eqref{eq010}.
\end{lem}
\begin{proof}
  Let us fix an infinite subset $M$ of $\nn$, a positive real
  $\ee$ and $l\in\nn$. Assume
  $\|x_s^i\|\mik C$ for all $s\in \ff$ and $i=1,...,k$.
  Let $\Lambda$ be a finite $\frac{\ee}{3lC}$-net of $[-1,1]$
  and $((a^q_j)_{j=1}^l)_{q=1}^n$
   an enumeration of all the $l$-tuples consisting of elements from
  $\Lambda$, where $n=|\Lambda|^l$.

  Setting $L_0=M$, we inductively construct a decreasing
  sequence $(L_q)_{q=0}^n$ of infinite subsets of $\nn$ such that
  for every $q=1,...,n$ and every
  $(s_j)_{j=1}^l,(t_j)_{j=1}^l$ from $\mathrm{{Plm}}_l(\ff\upharpoonright L_q)$,
  we have that
  \begin{equation}
    \label{eq013}
    \Bigg|\Big\|\sum_{j=1}^l a^q_jx^{i_k(j)}_{s_j}\Big\|-\Big\|\sum_{j=1}^l a^q_jx^{i_k(j)}_{t_j}\Big\|\Bigg|<\frac{\ee}{3}.
  \end{equation}
  The inductive step consists of an application of
  Theorem \ref{ramseyforplegma}. Indeed, assume that for some $1\mik q\mik n$ the set $L_{q-1}$
  has been chosen. Let $(A_r)_{r=1}^p$ be a partition of $[0,lC]$, such that
  $A_r$ is of diameter at most $\ee/3$ for all $r=1,...,p$.
  Observe that for every $(s_j)_{j=1}^l$ in $\mathrm{{Plm}}_l(\ff\upharpoonright L_q)$
  the vector
  $\sum_{j=1}^la^q_jx^{i_k(j)}_{s_j}$ is of norm at most $lC$.
  Thus setting, for every $r=1,...,p$, $\mathcal{P}_r$ to be the set
  of all $(s_j)_{j=1}^l$ from $\mathrm{{Plm}}_l(\ff\upharpoonright L_{q-1})$
  such that the norm of the vector
  $\sum_{j=1}^la^q_jx^{i_k(j)}_{s_j}$
  belongs to $A_r$, so $(\mathcal{P}_r)_{r=1}^p$ forms a
  partition of
  $\mathrm{{Plm}}_l(\ff\upharpoonright L_{q-1})$. An application of
  Theorem \ref{ramseyforplegma} yields the desired
  $L_q$ and the proof of the inductive step is complete.

  We set $L=L_n$. Clearly,
   for every $q=1,...,n$ and every
  $(s_j)_{j=1}^l,(t_j)_{j=1}^l$ from $\mathrm{{Plm}}_l(\ff\upharpoonright L)$,
  we have
  \begin{equation}
    \label{eq014}
    \Bigg|\Big\|\sum_{j=1}^la^q_jx^{i_k(j)}_{s_j}\Big\|-\Big\|\sum_{j=1}^la^q_jx^{i_k(j)}_{t_j}\Big\|\Bigg|<\frac{\ee}{3}.
  \end{equation}
  It remains to show that $L$ is as desired. Indeed, let $(s_j)_{j=1}^l,(t_j)_{j=1}^l$ in $
  \mathrm{{Plm}}_{kl}(\ff\upharpoonright L)$ and
  $(a_j)_{j=0}^l$ from $[-1,1]$.
  Pick $q_0\in\{1,...,n\}$
  such that $|a_j-a_j^{q_0}|\mik\ee/lC$, for all $j=1,...,l$.
  By the triangle inequality and the choice of $C$,
  \begin{equation}
    \label{eq015}
    \Bigg|\Big\|\sum_{j=1}^l a_jx^{i_k(j)}_{s_j}\Big\|-\Big\|\sum_{j=1}^l
    a^{q_0}_j x^{i_k(j)}_{s_j}\Big\|\Bigg|<\frac{\ee}{3}
  \end{equation}
  and
  \begin{equation}
    \label{eq016}
    \Bigg|\Big\|\sum_{j=1}^l a_jx^{i_k(j)}_{t_j}\Big\|-\Big\|\sum_{j=1}^l
    a^{q_0}_j x^{i_k(j)}_{t_j}\Big\|\Bigg|<\frac{\ee}{3}.
  \end{equation}
  Inequalities \eqref{eq014}-\eqref{eq016} yield
  \begin{equation}
    \label{eq017}
    \Bigg|\Big\|\sum_{j=1}^l a_jx^{i_k(j)}_{s_j}\Big\|-
    \Big\|\sum_{j=1}^l a_jx^{i_k(j)}_{t_j}\Big\|\Bigg|<\ee
  \end{equation}
  and the proof is complete.
\end{proof}
By iterating the above Lemma and diagonalizing, we obtain the following.

\begin{thm}
\label{existence_joint_sm}
  Let $X$ be a Banach space and $\ff$ a regular thin family.
  Also let $k\in\nn$ and $(x_s^1)_{s\in\ff},...,(x_s^k)_{s\in\ff}$ be
  bounded $\ff$-sequences in $X$. Then for every infinite subset
  $M$ of $\nn$ there exists a further infinite subset $L$ of $M$ such that
  the $k$-tuple $((x^i_s)_{s\in\ff\upharpoonright L})_{i=1}^k$ generates
  a joint $\ff$-model.
\end{thm}

We proceed to the following analogue of Theorem 6.11 from \cite{AKT2} for joint models.

\begin{thm}
  \label{unconditional_joint_spr}
  Let $X$ be a Banach space, $\ff$ a regular thin family and $k\in\nn$.
  Also let $M$ be an infinite subset of $\nn$ and $(x_s^1)_{s\in\ff},...,(x_s^k)_{s\in\ff}$
  seminormalized $\ff$-sequences in $X$ such that the $\ff$-subsequences
  $(x_s^1)_{s\in\ff\upharpoonright M},...,(x_s^k)_{s\in\ff\upharpoonright M}$
  are subordinated and weakly null. Also assume that the $k$-tuple
  $((x_s^i)_{s\in\ff\upharpoonright M})_{i=1}^k$ generates a joint $\ff$-model
  $(e_n)_n$. Then $(e_n)_n$ is (suppression) 1-unconditional.
\end{thm}
The proof of Theorem \ref{unconditional_joint_spr} follows similar lines as
the proof of Theorem 6.11 from \cite{AKT2}.
Let $l\in\nn$ and $F_1,...,F_l$ be subsets of $[\nn]^{<\infty}$. Recall that
$(F_j)_{j=1}^l$ is \emph{completely plegma connected} if for every choice $s_j\in F_j$ for all
$1\mik j\mik l$,  the $l$-tuple $(s_j)_{j=1}^l$ is a plegma family. Moreover,
the convex hull of a subset $A$ of a Banach space is denoted by $\mathrm{conv}(A)$.
For the proof of Theorem \ref{unconditional_joint_spr} we need to recall
Lemma 6.10 from \cite{AKT2}.

\begin{lem}\label{Lemma_finding_convex_means}
  Let $X$ be a Banach space, $l$ a positive integer, $\ff_1,\ldots,\ff_l$ regular thin families
 and $L$ an infinite subset of $\nn$.  Assume that   for every $i=1,...,l$,
 there exists  a continuous map  $\widehat{\varphi}_i:\widehat{\ff_i}\upharpoonright L\to X$, where $X$ is considered with the weak topology.
 Then for every   $\varepsilon>0$ there exists a completely plegma connected family
  $(F_i)_{i=1}^l$  such that    $F_i\subseteq [\ff_i\upharpoonright L]^{<\infty}$
  and
  $\mathrm{dist}\Big(\widehat{\varphi}_i(\varnothing), \mathrm{conv
  }\widehat{\varphi}_i(F_i)\Big)<\varepsilon,$
  for every $i=1,...,l$.
\end{lem}

\begin{proof}[Proof of Theorem \ref{unconditional_joint_spr}]
Fix
$l\in\nn$, some $1\mik p\mik l$ and $a_1,\ldots,a_l$ in $[-1,1]$. It
suffices to show  that for every $\varepsilon >0$ we have
\begin{equation}
\label{eq018}
\Big{\|}\sum_{\substack{j=1\\j\neq
p}}^l a_je_j\Big{\|}_*<\Big{\|}\sum_{j=1}^la_je_j\Big{\|}_*+\varepsilon.
\end{equation}
Fix $\varepsilon>0$. Since
$((x^i_s)_{s\in\ff\upharpoonright M})_{i=1}^k$
generates $(e_n)_n$ as a joint $\ff$-model, by passing to a
tail of $M$ if it is necessary, we may assume that
\begin{equation}
\label{eq019}
\Bigg{|} \Big{\|}\sum_{\substack{j=1\\j\neq p}}^l a_j
x^{i_k(j)}_{s_j}\Big{\|} -\Big{\|}
\sum_{\substack{j=1\\j\neq p}}^l a_je_j\Big{\|}_* \Bigg{|}<\frac{\varepsilon}{3}
\;\;\text{and}\;\; \Bigg{|}\Big{\|}\sum_{j=1}^l a_j
x^{i_k(j)}_{s_j}\Big{\|}-\Big{\|}\sum_{j=1}^l a_je_j\Big{\|}_*
\Bigg{|}<\frac{\varepsilon}{3},
\end{equation}
 for every plegma $l$-tuple $(s_j)_{j=1}^l$ in $\ff\upharpoonright
M$, where $i_k$ is as defined in \eqref{eq010}.
The first inequality in \eqref{eq019} is obtained by setting $a_p$
to be $0$.
Since for every $1\mik i\mik k$ the $\ff$-subsequence
$(x^i_s)_{s\in\ff\upharpoonright M}$ is subordinated, there exists a continuous map
$\widehat{\varphi}_i:\widehat{\ff}\upharpoonright M\to X$ such that
$\widehat{\varphi}_i(s)=x^i_s$ for every $s\in\ff\upharpoonright M$.
Moreover, for every $1\mik i\mik k$,
since $(x^i_s)_{s\in\ff\upharpoonright M}$ is weakly convergent to
$\widehat{\varphi}_i(\varnothing)$ and by assumption weakly null,
we have that
$\widehat{\varphi}_i(\varnothing)=0$. Therefore by
Lemma \ref{Lemma_finding_convex_means} (for ``$\ff_j=\ff$'' and
``$\widehat{\varphi}_j=\widehat{\varphi}_{i_k(j)}$'', for all $j=1,...,l$),
there exist a completely plegma connected family $(F_j)_{j=1}^l$
and a sequence $(x_j)_{j=1}^l$ in $X$
such that $F_j$ is subset of $[\ff\upharpoonright M]^{<\infty}$,
$x_j\in \mathrm{conv}\
\widehat{\varphi}_{i_k(j)}(F_j)$ and $\| x_j\|<\varepsilon/3$, for all
$1\mik j\mik l$. Let $(\mu_s)_{s\in F_p}$ be a sequence in $[0,1]$
such that $\sum_{s\in F_p}\mu_s=1$ and $x_p=\sum_{s\in
F_p}\mu_s\widehat{\varphi}_{i_k(p)}(s)$ and for each $j\neq p$ choose
$s_j\in
F_j$.
Clearly $\|x_p\|=\|\sum_{s\in F_p}\mu_s x_s\|<\frac{\varepsilon}{3}$.
Since $(F_j)_{j=1}^l$ is completely plegma connected we have
for every $s$ in $F_p$ that the $l$-tuple
$(s_1,\ldots,s_{p-1},s,s_{p+1},\ldots, s_l)$ is a plegma family.
Therefore by  \eqref{eq019} we have
  \begin{equation}
  \label{eq020}
  \begin{split}
  \Big{\|}\sum_{\substack{j=1\\j\neq p}}^la_je_j\Big{\|}_* &
  \mik\Big{\|}\sum_{\substack{j=1\\j\neq p}}^la_jx^{i_k(j)}_{s_j}\Big{\|}
  +\frac{\varepsilon}{3}
  \mik \Big{\|}\sum_{\substack{j=1\\j\neq p}}^la_jx^{i_k(j)}_{s_j}
  +a_px_p\Big{\|}+
  |a_p|\frac{\varepsilon}{3}+\frac{\varepsilon}{3}\\
  &\mik \Big{\|}\sum_{\substack{j=1\\j\neq p}}^la_jx^{i_k(j)}_{s_j}
  +a_p\sum_{s\in F_p}\mu_s x^{i_k(p)}_s\Big{\|}+\frac{2\varepsilon}{3}\\
  &\mik\sum_{s\in F_p}\mu_s \Big{\|}\sum_{\substack{j=1\\j\neq p}}^la_jx^{i_k(j)}_{s_j}
  +a_p x^{i_k(p)}_s\Big{\|}+\frac{2\varepsilon}{3}\\
  &\mik \sum_{s\in F_p}\mu_s \Big{(}\Big{\|}\sum_{j=1}^l a_je_j \Big{\|}_*
  +\frac{\varepsilon}{3} \Big{)}+\frac{2\varepsilon}{3} =\Big{\|}\sum_{j=1}^l a_je_j \Big{\|}_* +\varepsilon.
  \end{split}
  \end{equation}
  The proof is complete.
\end{proof}
The following lemma allows us to establish the semi-lattice structure of $SM_\xi^w(X)$.
\begin{lem}
  \label{lem_semi_lattice}
  Let $X$ be a Banach space and $\xi $  a countable ordinal.
  Also let $(e_n^1)_n,...,(e_n^k)_n$ be elements of $SM_\xi^w(X)$. Then there exists $(e_n)_n$ in $SM_\xi^w(X)$
  such that
  \begin{equation}
  \label{eq021}
 \max_{1\mik i\mik k}\Big\|\sum_{j=1}^la_je_j^i\Big\|\mik
  \Big\|\sum_{j=1}^la_je_j\Big\|
  \mik \sum_{i=1}^k \Big\|\sum_{j=1}^l a_je^i_j \Big\|\;\;\;\;
  \Big(\mik k\cdot\max_{1\mik i\mik k}\Big\|\sum_{j=1}^la_je_j^i\Big\|\Big)
\end{equation}
for every choice of $l\in \nn$ and $a_1,...,a_l\in\rr$.
\end{lem}
Before we proceed to the proof of Lemma \ref{lem_semi_lattice}
let us recall some notation. A family of finite subsets $\mathcal{H}$ of $\nn$ is called \emph{large}
(resp. \emph{very large}) in an infinite subset $M$
of $\nn$, if every further infinite subset $L$ of $\nn$ contains an element
(resp. has an initial segment) in $\mathcal{H}$.
It is immediate that both the notions of large and very large are hereditary,
i.e. if $\mathcal{H}$ is large (resp. very large) in some $M$
then  $\mathcal{H}$ is large (resp. very large) in any $L$ infinite
subset of $M$. We will need the following well known result due to F. Galvin
and K. Prikry \cite{GP} (it is actually a reformulation provided in
\cite{Go}).
\begin{thm} \label{Galvin prikry}
Let $\mathcal{H}$ be a family of finite subsets of $\nn$ and $M$ an
infinite subset of $\nn$. If $\mathcal{H}$
is large in $M$ then there exists an infinite subset $L$ of $M$ such that $\mathcal{H}$
is very large in $L$.
\end{thm}
It is easy to see that every regular thin family $\ff$ is
large in $\nn$. Thus by the above Theorem we have the following.
\begin{cor}
  \label{large_F}
  Let $\ff$ be a regular thin family and $M$ an infinite subset of $\nn$.
  Then there exists an infinite subset $L$ of $M$ such that $\ff$ is very large in
  $L$.
\end{cor}


\begin{proof}[Proof of Lemma \ref{lem_semi_lattice}]
  Let $\ff$ be a regular thin family of order $\xi$.
  Applying Corollary \ref{generating_subord_sm}, we obtain an infinite subset $M_0$
  of $\nn$ and $\ff$-sequences $(x^1_s)_{s\in\ff},...,(x^k_s)_{s\in\ff}$ in $X$ such that for every
  $1\mik i\mik k$ the $\ff$ subsequence $(x^i_s)_{s\in\ff\upharpoonright M_0}$ is subordinated, weakly null
  and generates $(e_n^i)_n$ as an $\ff$-spreading model.
  By Corollary \ref{large_F}, we may assume that $\ff$ is very large in $M_0$.
  Using Theorem \ref{existence_joint_sm},
  we pass to an infinite subset $M$ of $M_0$ such that the $k$-tuple
  $((x^i_s)_{s\in\ff\upharpoonright M})_{i=1}^k$ generates
  a joint $\ff$-model $(v_n)_n$. Theorem \ref{unconditional_joint_spr} yields that
  $(v_n)_n$ is (suppression) 1-unconditional.  We pick a sequence $(F_n)_n$ of finite subsets of $M$
  such that $\max F_n<\min F_{n+1}$ and $F_n$ is of cardinality $k$ for all $n\in\nn$.
  We set $N=\{\max F_n:n\in\nn\}$. Clearly $N$ is an infinite subset of $M$. For every
  $s\in\ff\upharpoonright N$ and every $1\mik i\mik k$ we set $t^i_s$ to be the unique element in $\ff$ being
  initial segment of $\{F_{s(q)}(i):1\mik q\mik |s|\}$. Observe that the existence of $t_s^i$ is guaranteed
  by the spreading property of $\widehat{\ff}$ and the fact that $\ff$ is very large, while its uniqueness is a consequence of the fact that $\ff$ is
  thin. Also observe that for every $l\in\nn$ and every plegma family $\mathbf{s}=(s_j)_{j=1}^l$ in $\ff\upharpoonright N$ of length $l$, setting $(t^{\mathbf{s}}_q)_{q=1}^{kl}=(t^1_{s_1},t^2_{s_1},...,t^k_{s_1},...,t^1_{s_l},t^2_{s_l},...,t^k_{s_l})$,
  we have that $(t^{\mathbf{s}}_q)_{q=1}^{kl}$ is a plegma family in
  $\ff\upharpoonright M$. For every $s\in\ff\upharpoonright N$ we set
  $z_s=\sum_{i=1}^kx^i_{t_s^i}$. Pass to an infinite subset $L$ of $N$ such that
  $(z_s)_{s\in\ff\upharpoonright L}$ generates an $\ff$-spreading model $(e_n)_n$.
  The following claim holds.

  \noindent\textbf{Claim:} Let $l\in\nn$, $a_1,...,a_l$ in $[-1,1]$ and $\ee>0$. Then
  \begin{equation}
  \label{eq022}
  \max_{1\mik i\mik k}\Big\|\sum_{j=1}^la_je_j^i\Big\|-\ee\mik
  \Big\|\sum_{j=1}^la_je_j\Big\|
  \mik \sum_{i=1}^k \Big\|\sum_{j=1}^l a_je^i_j \Big\|+\frac{\ee}{2}.
  \end{equation}
  \begin{proof}[Proof of Claim]
  Let $(b_q)_{q=1}^{kl}$ defined by $b_{(j-1)k+i}=a_j$ for all $1\mik j\mik l$ and $1\mik i\mik k$. We pass to a final segment $L'$ of $L$ such that
  for every plegma family $\mathbf{s}=(s_j)_{j=1}^l$ in $\ff\upharpoonright L'$ we have that
  \begin{equation}
    \label{eq023}
    \Bigg|\Big\|\sum_{j=1}^l a_je_j \Big\|-\Big\|\sum_{j=1}^l a_jz_{s_j}  \Big\|\Bigg|<\ee/4,
  \end{equation}
  \begin{equation}
    \label{eq024}
    \Bigg|\Big\|\sum_{j=1}^l a_je^i_j \Big\|-\Big\|\sum_{j=1}^l a_jx^i_{t^i_{s_j}}  \Big\|\Bigg|<\ee/4k
  \end{equation}
  for all $1\mik i\mik k$ and
  \begin{equation}
    \label{eq025}
    \Bigg|\Big\|\sum_{q\in F} b_qv_q \Big\|-\Big\|\sum_{q\in F} b_q x^{i_k(q)}_{t^{\mathbf{s}}_q}  \Big\|\Bigg|<\ee/4
  \end{equation}
  for all $F\subseteq\{1,...,kl\}$. We set $F_i=\{(j-1)k+i:j=1,...,l\}$ for all
  $1\mik i\mik k$. Let us also fix some plegma family $\mathbf{s}=(s_j)_{j=1}^l$ in $\ff\upharpoonright L'$. Recall that by Theorem \ref{unconditional_joint_spr} it follows that
  $(v_n)_n$ is (suppression) 1-unconditional. Hence for every $1\mik i\mik k$, we have
  \begin{equation}
    \label{eq026}
    \begin{split}
      \Big\|\sum_{j=1}^la_je_j^i\Big\|-\ee&
      \stackrel{\eqref{eq024}}{\mik}\Big\|\sum_{j=1}^l a_jx^i_{t^i_{s_j}}\Big\|-\frac{3\ee}{4}
      =\Big\|\sum_{q\in F_i} b_q x^{i_k(q)}_{t^{\mathbf{s}}_q}\Big\|-\frac{3\ee}{4}\\
      &\stackrel{\eqref{eq025}}{\mik}\Big\|\sum_{q\in F_i} b_qv_q \Big\|-\frac{\ee}{2}
      \mik \Big\|\sum_{q=1}^{kl} b_qv_q \Big\|-\frac{\ee}{2}\\
      &\stackrel{\eqref{eq025}}{\mik}\Big\|\sum_{q=1}^{kl} b_q x^{i_k(q)}_{t^{\mathbf{s}}_q}\Big\|-\frac{\ee}{4}
      = \Big\|\sum_{j=1}^l a_jz_{s_j}\Big\|-\frac{\ee}{4}
      \stackrel{\eqref{eq023}}{\mik}\Big\|\sum_{j=1}^l a_je_j \Big\|.
    \end{split}
  \end{equation}
  Since \eqref{eq026} holds for all $1\mik i\mik k$ we get
  \begin{equation}
    \label{eq027}
    \max_{1\mik i\mik k}\Big\|\sum_{j=1}^la_je_j^i\Big\|-\ee\mik
    \Big\|\sum_{j=1}^la_je_j\Big\|.
  \end{equation}
  Making use of the triangle inequality we have that
\begin{equation}
    \label{eq028}
    \begin{split}
      \Big\|\sum_{j=1}^l a_je_j \Big\|
      &\stackrel{\eqref{eq023}}{\mik}\Big\|\sum_{j=1}^l a_jz_{s_j}\Big\|+\frac{\ee}{4}
      =\Big\|\sum_{q=1}^{kl} b_q x^{i_k(q)}_{t^{\mathbf{s}}_q}\Big\|+\frac{\ee}{4}\\
      &\mik \sum_{i=1}^k \Big\|\sum_{q\in F_i} b_q x^{i_k(q)}_{t^{\mathbf{s}}_q}\Big\|+\frac{\ee}{4}
      =\sum_{i=1}^k \Big\|\sum_{j=1}^l a_j x^{i_k(F_i(j))}_{t^{\mathbf{s}}_{F_i(j)}}\Big\|+\frac{\ee}{4}\\
      &=\sum_{i=1}^k \Big\|\sum_{j=1}^l a_j x^i_{t^{\mathbf{s}}_{F_i(j)}}\Big\|+\frac{\ee}{4}
      \stackrel{\eqref{eq024}}{\mik}\sum_{i=1}^k \Big\|\sum_{j=1}^l a_je^i_j \Big\|+\frac{\ee}{2}.
    \end{split}
  \end{equation}
  Clearly \eqref{eq022} follows by \eqref{eq027} and \eqref{eq028}.
\end{proof}
\noindent By the Claim above the result is immediate.
\end{proof}

The above lemma has the following immediate consequence.
\begin{thm}\label{lattice}
  Let $X$ be a Banach space and $\xi $ a countable ordinal.
  Then $\mathbf{SM}_\xi^w(X)$ is a (upper) semi-lattice.
\end{thm}


\begin{thm}\label{lattice_countable}
  Let $X$ be a Banach space and $\xi$  a countable ordinal.
  Let $(c_k)_k$ be a sequence of positive reals satisfying $\sum_{k=1}^\infty c_k^{-1}<\infty$
  and for every $k\in\nn$
  let $(e_n^k)_n$ be a normalized sequence that belongs to $SM_\xi^w(X)$.
  Then there exist $(e_n)_n$ in $SM_\xi^w(X)$ and a real $K$ with  $\max_{k\in\nn}c_k^{-1}\mik K\mik\sum_{k=1}^\infty c_k^{-1}$ satisfying the following.
  \begin{enumerate}
    \item[(i)] The sequence $(e_n)_n$ is normalized.
    \item[(ii)] The sequence $(e_n)_n$  $(c_kK)$-dominates $(e^k_n)_n$,
  for all $k\in\nn$.
    \item[(iii)] For every $l\in\nn$ and every choice of reals $a_1,...,a_l$ in $[-1,1]$ we have that $\|\sum_{j=1}^la_je_j\|\mik K^{-1}\sum_{k=1}^\infty c_k^{-1}\|\sum_{j=1}^la_je^k_j\|$.
  \end{enumerate}
\end{thm}
\begin{proof}
  Let $\ff$ be a regular thin family of order $\xi$.
  Applying Corollary \ref{large_F} we obtain an infinite subset $M_0$ of $\nn$ such that
  $\ff$ is very large in $M_0$.
  Using
  Corollary \ref{generating_subord_sm}, we inductively obtain a decreasing sequence $(M'_k)_k$
  of infinite subsets of $M_0$ and for every $k\in\nn$ a seminormalized $\ff$-sequence
  $(x^k_s)_{s\in\ff}$ such that for every $k\in\nn$ the $\ff$-subsequence
  $(x^k_s)_{s\in\ff\upharpoonright M'_k}$
  is subordinated, weakly null and generates $(e_n^k)_n$
  as an $\ff$-spreading model. Moreover, setting $B_k=\sup\{\|x_s^k\|:s\in\ff\upharpoonright M'_k\}$,
  we may assume that $B_k\mik2$, for all $k\in\nn$.
  Hence
  $\sum_{k=1}^\infty c_k^{-1}B_k<\infty$.
  For every $k\in\nn$ and
  $s\in \ff$ we set $y_s^k=c_k^{-1}x_s^k$. Clearly for every $k\in\nn$ the $\ff$-sequence
  $(y^k_s)_{s\in\ff}$ is seminormalized, while the $\ff$-subsequence
  $(y^k_s)_{s\in\ff\upharpoonright M'_k}$
  is subordinated, weakly null and generates $(c_k^{-1}e_n^k)_n$
  as an $\ff$-spreading model. Let $M'$ be an infinite subset of
  $M_0$ such that $M'(k)\in M'_k$ for all $k\in\nn$. Using
  Theorem \ref{existence_joint_sm} we obtain a decreasing sequence
  $(M_k)_k$ of infinite subsets of $M'$ such that for every $k\in\nn$ the
  $k$-tuple $((y^i_s)_{s\in\ff\upharpoonright M_k})_{i=1}^k$ generates
  a joint $\ff$-model $(v_n^k)_n$. Observe that, by Theorem
  \ref{unconditional_joint_spr}, the sequence $(v_n^k)_n$ is (suppression)
  1-unconditional for all $k\in\nn$. Let $M$ be an infinite subset
  of $M'$ such that $M(k)\in M_k$ for all $k\in\nn$. Fix a sequence $(F_n)_n$ of finite subsets
  of $M$ such that for every $n\in\nn$ $\max F_n<\min F_{n+1}$ and
  $F_n$ is of cardinality $n$. Set $L=\{\max F_n: n\in\nn\}$.
  For every $s\in\ff\upharpoonright L$ let $n_s$ to be
  such that $\min s=\max F_{n_s}$ and $t_s^i$ to be the unique initial
  segment of $\{F_{s(j)}(i):j=1,...,|s|\}$ belonging to $\ff$, for
  all $1\mik i\mik n_s$. The existence of $t_s^i$
  follows by the fact that $\ff$ is very large in $L$ and $\widehat{\ff}$ is spreading,
  while the uniqueness of $t_s^i$ follows by the fact that $\ff$ is thin.
  Observe that for every $s\in\ff\upharpoonright L$ and every $1\mik i\mik n_s$,
  $\min t^i_s\meg M(i)\meg M'(i)\in M'_i$ and therefore
  \begin{equation}
    \label{eq029}
    \|x^i_{t_s^i}\|\mik B_i.
  \end{equation}
  For every $s$ in $\ff\upharpoonright L$ set
  $z_s=\sum_{i=1}^{n_s}y^i_{t_s^i}$ and pass to an infinite subset
  $L'$ of $L$ such that the $\ff$-subsequence
  $(z_s)_{s\in\ff\upharpoonright L'}$ generates an $\ff$-spreading model
  $(e'_n)_n$. Since $({c}_k^{-1} B_k)_k$ is summable and for
  each $k\in\nn$ the $\ff$-subsequence $(x_s^k)_{s\in\ff\upharpoonright L'}$
  is subordinated and weakly null, by \eqref{eq029}, one can easily derive that
  $(z_s)_{s\in\ff\upharpoonright L'}$ is also subordinated and
  weakly null. Thus $(e'_n)_n$ belongs to $SM_\xi^w(X)$.
  First we prove the following claim.

  \noindent\textbf{Claim:}
  The sequence $(e'_n)_n$ $c_k$-dominates $(e^k_n)_n$,
  for all $k\in\nn$.
  \begin{proof}[Proof of Claim]
  Let us fix some $k\in\nn$.
  Pick $l\in\nn$, real numbers $a_1,...,a_l$ in $[-1,1]$
  and $\ee>0$. We will show that
  \begin{equation}
    \label{eq030}
    \Big\|\sum_{j=1}^la_je^k_j\Big\|\mik c_k\Big\|\sum_{j=1}^la_je'_j\Big\|+\ee.
  \end{equation}
  Pick $k'\meg k$ such that
  \begin{equation}
    \label{eq031}
    \sum_{q=k'+1}^\infty c_q^{-1}B_q<\frac{\ee}{5c_k\sum_{j=1}^l|a_j|}.
  \end{equation}
  Let $(b_q)_{q=1}^{k'l}$ defined by $b_{(j-1)k'+i}=a_j$ for all $1\mik j\mik l$ and $1\mik i\mik k'$.
  Moreover, for every plegma family $\mathbf{s}=(s_j)_{j=1}^l$ in $\ff\upharpoonright L$ with $n_{s_1}\meg L(k')$ we set
  $(t_q^{\mathbf{s}})_{q=1}^{k'l}=(t^1_{s_1},...,t^{k'}_{s_1},...,t^1_{s_l},...,t^{k'}_{s_l})$.
  Observe that for every plegma
  family $\mathbf{s}=(s_j)_{j=1}^l$ in $\ff\upharpoonright L$ with $n_{s_1}\meg L(k')$ both the $k'l$-tuple $(t_q^{\mathbf{s}})_{q=1}^{k'l}$ and the
  $\big(\sum_{j=1}^ln_{s_j}\big)$-tuple
  $(t^1_{s_1},...,t^{n_{s_1}}_{s_1},...,t^1_{s_l},...,t^{n_{s_l}}_{s_l})$
  are plegma families.
  We pass to an infinite subset $L'$ of $L$ such that $\min L'\meg L(k')$ and for every plegma
  family $\mathbf{s}=(s_j)_{j=1}^l$ in $\ff\upharpoonright L'$  we have that
  \begin{equation}
    \label{eq032}
    \Bigg|\Big\|\sum_{j=1}^la_je'_j\Big\|-\Big\|\sum_{j=1}^la_jz_{s_j}\Big\|\Bigg|<\frac{\ee}{5c_k},
  \end{equation}
  \begin{equation}
    \label{eq033}
    \Bigg|c^{-1}_k\cdot\Big\|\sum_{j=1}^la_je^{k}_j\Big\|-\Big\|\sum_{j=1}^la_jy^k_{t^k_{s_j}}\Big\|\Bigg|<\frac{\ee}{5c_k},
  \end{equation}
  and
  \begin{equation}
    \label{eq034}
    \Bigg|\Big\|\sum_{q\in F}b_qy^{i_{k'}(q)}_{t^{\mathbf{s}}_q}\Big\|- \Big\|\sum_{q\in F}b_qv^{k'}_q\Big\|\Bigg| <\frac{\ee}{5c_k}
  \end{equation}
  for all $F\subseteq\{1,...,k'l\}$, where $i_{k'}$ is as defined in \eqref{eq010}.
  Fix a plegma family $\mathbf{s}=(s_j)_{j=1}^l$
  in $\ff\upharpoonright L'$ and set $F_k=\{(j-1)k'+k:j=1,...,l\}$.
  By the unconditionality of $(v_n^{k'})_n$ we get that
  \begin{equation}
    \label{eq035}
    \begin{split}
      \Big\|\sum_{q=1}^{k'l}b_q y^{i_{k'}(q)}_{t^{\mathbf{s}}_q}\Big\|
      &\stackrel{\eqref{eq034}}{\meg}\Big\|\sum_{q=1}^{k'l}b_q v^{k'}_q\Big\|-\frac{\ee}{5c_k}
      \meg \Big\|\sum_{q\in F_k}b_q v^{k'}_q\Big\|-\frac{\ee}{5c_k}\\
      &\stackrel{\eqref{eq034}}{\meg}\Big\|\sum_{q\in F_k}b_q y^{i_{k'}(q)}_{t^{\mathbf{s}}_q}\Big\|-\frac{2\ee}{5c_k}
      =\Big\|\sum_{j=1}^la_jy^k_{t^k_{s_j}}\Big\|-\frac{2\ee}{5c_k}\\
      &\stackrel{\eqref{eq033}}{\meg}
      c^{-1}_k\cdot\Big\|\sum_{j=1}^la_je^{k}_j\Big\|-\frac{3\ee}{5c_k}.
    \end{split}
  \end{equation}
  Moreover we have the following.
  \begin{equation}
    \label{eq036}
    \begin{split}
      c_k\Big\|\sum_{j=1}^la_je'_j\Big\|+\ee
      &\stackrel{\eqref{eq032}}{\meg}c_k\Big\|\sum_{j=1}^la_jz_{s_j}\Big\|+\frac{4\ee}{5}
      =c_k\Big\|\sum_{j=1}^la_j\sum_{i=1}^{n_{s_j}}y^i_{t_{s_j}^i}\Big\|+\frac{4\ee}{5}\\
      &\meg c_k\Big\|\sum_{j=1}^la_j\sum_{i=1}^{k'}y^i_{t_{s_j}^i}\Big\|
        -c_k\sum_{j=1}^l|a_j|\cdot \Big\|\sum_{i=k'+1}^{n_{s_j}}y^i_{t_{s_j}^i}\Big\|+\frac{4\ee}{5}\\
      &\meg c_k\Big\|\sum_{j=1}^la_j\sum_{i=1}^{k'}y^i_{t_{s_j}^i}\Big\|
        -c_k\sum_{j=1}^l|a_j| \sum_{i=k'+1}^{n_{s_j}}c_i^{-1}\Big\|x^i_{t_{s_j}^i}\Big\|+\frac{4\ee}{5}\\
      &\stackrel{\eqref{eq029},\eqref{eq031}}{\meg}c_k\Big\|\sum_{j=1}^la_j\sum_{i=1}^{k'}y^i_{t_{s_j}^i}\Big\|
        +\frac{3\ee}{5}
      =c_k\Big\|\sum_{q=1}^{k'l}b_q y^{i_{k'}(q)}_{t^{\mathbf{s}}_q}\Big\|+\frac{3\ee}{5}.
    \end{split}
  \end{equation}
  Clearly \eqref{eq030} follows by \eqref{eq035} and \eqref{eq036}.
  Since \eqref{eq030} holds for every choice of natural numbers $k,l$, real numbers $a_1,...,a_l$ in $[-1,1]$
  and $\ee>0$, the claim follows.
\end{proof}

Let $K=\|e'_1\|$ and $(e_n)_n=(K^{-1}e'_n)_n$.
  Then $(e_n)_n$ is a normalized sequence belonging to $SM_\xi^w(X)$, i.e.
  assertion (i) of the theorem is satisfied.
  By the claim, assertion (ii) of the theorem is immediate and $K=\|e_1\|\meg c_k^{-1}\|e_1^k\|=c_k^{-1}$ for all $k\in\nn$.
  Thus $K\meg\max_{k\in\nn}c_k^{-1}$. Finally by the definition of $(e'_n)_n$ and $(e_n)_n$ it is easy to check that
  assertion (iii) of the lemma is also true and $K\mik\sum_{k=1}^\infty c^{-1}_k$. The proof is complete.
\end{proof}

\section{From countable to uncountable increasing sequences of spreading models}\label{section_S}

Using identical arguments
as the ones used in the proof of Theorem 2.2 from \cite{Sa} one can prove the following.

\begin{thm}\label{from_countable_to_uncountable}
  Let $\mathcal{C}$ be a family of normalized Schauder basic sequences satisfying the following.
  \begin{enumerate}
    \item[(i)] For every $(x_n)_n$ and $(y_n)_n$ in $\mathcal{C}$ there exists $(z_n)_n$ in $C$ such that $(z_n)_n$ $2$-dominates both $(x_n)_n$ and $(y_n)_n$.
    \item[(ii)] For every sequence $(c_k)_k$ of positive reals satisfying $\sum_{k=1}^\infty c_k^{-1}<\infty$
  and every infinite sequence $(x^1_n)_n,(x^2_n)_n,...$ in $\mathcal{C}$ there exist $(x_n)_n$ in $\mathcal{C}$ and a constant $K$ with  $\max_{k\in\nn}c_k^{-1}\mik K\mik\sum_{k=1}^\infty c_k^{-1}$ satisfying the following.
  \begin{enumerate}
    \item[(a)] The sequence $(x_n)_n$  $(c_kK)$-dominates $(x^k_n)_n$,
  for all $k\in\nn$.
    \item[(b)] For every $l\in\nn$ and every choice of reals $a_1,...,a_l$ in $[-1,1]$ we have that $\|\sum_{j=1}^la_jx_j\|\mik K^{-1}\sum_{k=1}^\infty c_k^{-1}\|\sum_{j=1}^la_jx^k_j\|$.
  \end{enumerate}
  \end{enumerate}
  If $\mathcal{C}$ contains a strictly increasing, with respect to domination, sequence of length $\omega$, then $\mathcal{C}$ contains a
  strictly increasing sequence of length $\omega_1$.
\end{thm}

By Lemma \ref{lem_semi_lattice} and Theorem \ref{lattice_countable} the collection $\mathcal{C}$ of the normalized elements of $SM_\xi^w(X)$ satisfies the assumptions of Theorem \ref{from_countable_to_uncountable}. Hence we have the following.
\begin{cor}\label{long_chains}
  Let $X$ be a Banach space and $\xi$ a countable ordinal.
  If $SM_\xi^w(X)$ contains a strictly increasing sequence of length $\omega$, then $\mathcal{C}$ contains a
  strictly increasing sequence of length $\omega_1$.
\end{cor}

\section{On the richness of $SM_\xi^w(X)$} \label{section_D}
In this section we will generalize some results from \cite{Do}.
We will code the set $SM_\xi^w(X)$, for $X$ separable, as an analytic subset of $[\nn]^\infty$. Recall that a binary relation $\preccurlyeq$ on some set $A$ is called a pre-partial order if the following are satisfied.
\begin{enumerate}
  \item[(i)] For every $a\in A$, we have that $a\preccurlyeq a$ and
  \item[(ii)] for every $a,b,c\in A$ with $a\preccurlyeq b$ and $b\preccurlyeq c$, we have that $a\preccurlyeq c$.
\end{enumerate}
As usual, every binary relation on a set $A$ can be viewed as a subset of the Cartesian $A\times A$. In the case that $A$ is a topological space, we endow the Cartesian product with the product topology of $A$.

\begin{prop}\label{descriptive}
  Let $\preccurlyeq$ be an $F_\sigma$ pre-partial order on $[\nn]^\infty$ and $\thickapprox$ the equivalence relation defined by $a\thickapprox b$ iff $a\preccurlyeq b$ and $b\preccurlyeq a$. Let $A$ be an analytic subset of
  $[\nn]^\infty$ such that either $A$ does not contain any strictly increasing sequence of length $\omega$,
  or $A$ contains a strictly increasing sequence of length $\omega_1$. Then the following hold.
  \begin{enumerate}
    \item[(i)] If $A/_\thickapprox$ is uncountable then $A$ contains an antichain of size the continuum, i.e. there exists a subset $P\subset A$ of cardinality $\mathfrak{c}$ such that for every $a\neq b$ in $P$ we have that
        $a\not\preccurlyeq b$ and $b\not\preccurlyeq a$.
    \item[(ii)] If $A$ contains a strictly decreasing sequence of
        length $\omega_1$, then $A$ contains a strictly increasing sequence of length $\omega_1$.
    \item[(iii)] If $A$ does not contain a strictly increasing sequence of length $\omega$, then there exists a countable ordinal $\zeta$ such that $A$ does not contain any decreasing sequence of length $\zeta$.
  \end{enumerate}
\end{prop}
The first assertion of the above proposition is a consequence of a result due to J.H. Silver \cite{Si} (see also Lemma 5 from \cite{Do} for simplified version adapted to our needs). Actually, by Silver's Theorem the set $P$ can be chosen to be a perfect subset of $A$. Assertions (ii) and (iii) of the above proposition follow by similar arguments to the ones developed in the proof of Theorem 3 of \cite{Do}.

\begin{thm}\label{thm_gen_D}
  Let $X$ be a Banach space with separable dual and $\xi$ a countable ordinal. Then the following hold true.
  \begin{enumerate}
    \item[(i)] If $\mathbf{SM}_\xi^w(X)$ is uncountable then there exist continuum many pairwise incomparable elements of $SM_\xi^w(X)$.
    \item[(ii)] If $SM_\xi^w(X)$ contains a strictly decreasing sequence of
        length $\omega_1$, then $SM_\xi^w(X)$ contains a strictly increasing sequence of length $\omega_1$.
    \item[(iii)] If $SM_\xi^w(X)$ does not contain a strictly increasing sequence of length $\omega$, then there exists a countable ordinal $\zeta$ such that $SM_\xi^w(X)$ does not contain any decreasing sequence of length $\zeta$.
  \end{enumerate}
\end{thm}
The above theorem follows by Proposition \ref{descriptive}
and the following analogue of Lemma 7 from \cite{Do} which provides the desirable coding of $SM_\xi^w(X)$. In order to state it we need some additional notation. Let $(u_n)_n$ be the standard unconditional Schauder basis of the universal space of Pe{\l}czy\'nski for the unconditional Schauder basic sequences (see \cite{Pe}). We define the following pre-partial ordering $\preccurlyeq$ on $[\nn]^\infty$. For every $L$ and $M$ in $[\nn]^\infty$ we set $L\preccurlyeq M$ iff $(u_n)_{n\in M}$ dominates $(u_n)_{n\in L}$.
\begin{lem}\label{lem_cod_1}
  Let $X$ be a Banach space with separable dual and $\xi$ be a countable ordinal. Then there exists an analytic subset $A$ of $[\nn]^\infty$
  satisfying the following.
  \begin{enumerate}
    \item[(i)] For every $(e_n)_n$ in $SM_\xi^w(X)$ there exists $M$ in $A$ such that the sequences $(e_n)_n$ and $(u_n)_{n\in M}$ are equivalent.
    \item[(ii)] For every $M$ in $A$  there exists $(e_n)_n$ in $SM_\xi^w(X)$ such that the sequences $(e_n)_n$ and $(u_n)_{n\in M}$ are equivalent.
  \end{enumerate}
\end{lem}
\begin{proof}
  Fix a regular thin family $\ff$ of order $\xi$.  We consider the following subset $G$ of $[\nn]^\infty\times [\nn]^\infty\times X^\ff\times X^{\widehat{\ff}}$. We set $(M, L,(x_s)_{s\in\ff},(y_t)_{t\in\widehat{\ff}})\in G$
  if the following are satisfied:
  \begin{enumerate}
    \item[(a)] There exists $C>0$ such that for every $k\in\nn$, every $(s_j)_{j=1}^k$ in $\mathrm{Plm}_k(\ff\upharpoonright L)$ with $s_1(1)\meg L(k)$ and every $a_1,...,a_k$ reals we have
        \begin{equation}
          \label{eq037}
          C^{-1}\Big\|\sum_{j=1}^ka_jx_{s_j}\Big\|
          \mik\Big\|\sum_{j=1}^ka_ju_{M(j)}\Big\|
          \mik C\Big\|\sum_{j=1}^ka_jx_{s_j}\Big\|
        \end{equation}
        where $(u_n)_n$ is the standard unconditional Schauder basis of the universal space of Pe{\l}czy\'nski for the unconditional Schauder basic sequences.
    \item[(b)] The $\ff$-subsequence
        $(x_s)_{s\in\ff\upharpoonright L}$ is subordinated and weakly null. Moreover, if $\widehat{\varphi}:\widehat{\ff}\upharpoonright L\to X$ is the map witnessing the fact that
        $(x_s)_{s\in\ff\upharpoonright L}$ is subordinated , then $\widehat{\varphi}(t)=y_t$ for all $t\in\widehat{\ff}\upharpoonright L$.
  \end{enumerate}
  Invoking the separability of $X^*$, it is easy to check that $G$ is a Borel subset of $[\nn]^\infty\times [\nn]^\infty\times X^\ff\times X^{\widehat{\ff}}$. We set $A$ to be the projection of $G$ to the first coordinate, that is
  \begin{equation}
          \label{eq038}
          \begin{split}
            A=\{M\in[\nn]^\infty:\;&\text{there exists}\;(L,(x_s)_{s\in\ff},(y_t)_{t\in\widehat{\ff}})\in[\nn]^\infty\times X^{\ff}\times X^{\widehat{\ff}}\;\\
          &\text{such that}(M,L,(x_s)_{s\in\ff},(y_t)_{t\in\widehat{\ff}})\in G\;\}.
          \end{split}
        \end{equation}
  Since $G$ is Borel, we get that $A$ is analytic. It remains to check that $A$ satisfies (i) and (ii) of the lemma. Indeed, let $(e_n)_n$ in $SM_\xi^w(X)$. By Corollary \ref{generating_subord_sm}, there exist
  an infinite subset $L$ of $\nn$ and an $\ff$-sequence $(x_s)_{s\in\ff}$ in $X$ such that the $\ff$-subsequence $(x_s)_{s\in\ff\upharpoonright L}$ is subordinated, weakly null and generates $(e_n)_n$ as an $\ff$-spreading model. Moreover, by the universality of Pe{\l}czy\'nski's space, there exists an infinite subset $M$ of $\nn$ such that the sequences $(u_n)_{n\in M}$ and $(e_n)_n$ are equivalent. Finally, if $\widehat{\varphi}:\widehat{\ff}\to X$ is the continuous map witnessing the fact that $(x_s)_{s\in\ff\upharpoonright L}$ is subordinated, then we set $y_t=\widehat{\varphi}(t)$ for all $t\in \widehat{\ff}\upharpoonright L$ and $y_t=0$ otherwise. It follows readily that $(M,L, (x_s)_{s\in\ff},(y_t)_{t\in\widehat{\ff}})$ belongs to $G$ and therefore $M$ belongs to $A$. Since the sequences $(u_n)_{n\in M}$ and $(e_n)_n$ are equivalent, conclusion (i) of the lemma is satisfied.
  Conversely, let $M\in A$. By the definition of $A$, there exist an infinite subset $L$ of $\nn$, an $\ff$-sequence $(x_s)_{s\in\ff})$ in $X$ and a family $(y_t)_{t\in\widehat{\ff}}$ of elements in $X$ such that $(M, L,(x_s)_{s\in\ff},(y_t)_{t\in\widehat{\ff}})$ belongs to $G$. We pass to an infinite subset $L'$ of $L$ such that the $\ff$-subsequence $(x_s)_{s\in\ff\upharpoonright L'}$ generates an $\ff$-spreading model $(e_n)_n$. By \eqref{eq037}, it is easy to see that the sequences $(e_n)_n$ and $(u_n)_{n\in M}$ are equivalent, while by (b) above we have that $(e_n)_n$ belongs to $SM_\xi^w(X)$. That is, the conclusion (ii) of the lemma holds true and the proof is complete.
\end{proof}

A question of interest is whether one can drop the separable dual assumption in Theorem \ref{thm_gen_D}. Towards that direction we have the following.

\begin{thm}\label{thm_gen_D_2}
  Let $X$ be a separable Banach space admitting no spreading model of order 1 equivalent to the standard basis of $\ell^1$ and $\xi$ a countable ordinal. Then the following hold true.
  \begin{enumerate}
    \item[(i)] If $\mathbf{SM}_\xi^w(X)$ is uncountable then there exist continuum many pairwise incomparable elements of $SM_\xi^w(X)$.
    \item[(ii)] If $SM_\xi^w(X)$ contains a strictly decreasing sequence of
        length $\omega_1$, then $SM_\xi^w(X)$ contains a strictly increasing sequence of length $\omega_1$.
    \item[(iii)] If $SM_\xi^w(X)$ does not contain any strictly increasing sequence of length $\omega$, then there exists a countable ordinal $\zeta$ such that $SM_\xi^w(X)$ does not contain any decreasing sequence of length $\zeta$.
  \end{enumerate}
\end{thm}
Theorem \ref{thm_gen_D_2} follows by Proposition \ref{descriptive} and the following analogue of Lemma \ref{lem_cod_1} which provides us with the desirable analytic coding of the set $SM_\xi^w(X)$.

\begin{lem}\label{lem_cod_2}
  Let $X$ be a separable Banach space admitting no spreading model of order 1 equivalent to the standard basis of $\ell^1$ and $\xi$ a countable ordinal.
  Then there exists an analytic subset $A$ of $[\nn]^\infty$
  satisfying the following.
  \begin{enumerate}
    \item[(i)] For every $(e_n)_n$ in $SM_\xi^w(X)$ there exists $M$ in $A$ such that the sequences $(e_n)_n$ and $(u_n)_{n\in M}$ are equivalent.
    \item[(ii)] For every $M$ in $A$  there exists $(e_n)_n$ in $SM_\xi^w(X)$ such that the sequences $(e_n)_n$ and $(u_n)_{n\in M}$ are equivalent.
  \end{enumerate}
\end{lem}
Let us recall that $(u_n)_n$ is the standard unconditional Schauder basis of the universal space of Pe{\l}czy\'nski for the unconditional Schauder basic sequences (see \cite{Pe}).
For the proof of Lemma \ref{lem_cod_2} we will need the following lemma.
\begin{lem}
  \label{pushing_cont}
  Let $\ff$ be a regular thin family and $L$ an infinite subset of $\nn$. Also let $X$ be a Banach
  space and $\widehat{\varphi}:\widehat{\ff}\upharpoonright L\to X$ a map such that for every $t$ in $(\widehat{\ff}\upharpoonright L)\setminus \ff$ and every sequence $(s_n)_n$ in $\ff\upharpoonright L$ convergent to $t$, we have that $\widehat{\varphi}(s_n)\stackrel{w}{\to}\widehat{\varphi}(t)$. Then $\widehat{\varphi}$ is subordinated.
\end{lem}
\begin{proof}
  Since the topology on $\widehat{\ff}\upharpoonright L$ is metrizable, it suffices to check that $\widehat{\varphi}$ is sequentially continuous. Let $(t_n)_n$ be a convergent sequence in $\widehat{\ff}\upharpoonright L$ to some $t$. Clearly $t$ belongs to $\widehat{\ff}\upharpoonright L$. Moreover, without loss of generality, we may assume that $\min (t_n\setminus t)\to\infty$.
  We need to show that $\widehat{\varphi}(t_n)\stackrel{w}{\to}\widehat{\varphi}(t)$. Fix $x^*\in X^*$. For every $n\in\nn$, we pick $s_n\in \ff\upharpoonright L$ such that $t_n\sqsubseteq s_n$
  and $|x^*(\widehat{\varphi}(t_n))-x^*(\widehat{\varphi}(s_n))|<1/n$. Since $\min (t_n\setminus t)\to\infty$, we get that $(s_n)_n$ converges to $t$. By the assumptions on the map $\widehat{\varphi}$, we have that $x^*(\widehat{\varphi}(s_n))\to x^*(\widehat{\varphi}(t))$. Since $|x^*(\widehat{\varphi}(t_n))-x^*(\widehat{\varphi}(s_n))|<1/n$ for all $n\in\nn$, we get that
  $x^*(\widehat{\varphi}(t_n))\to x^*(\widehat{\varphi}(t))$ and the proof is complete.
\end{proof}
Before we proceed to the proof of Lemma \ref{lem_cod_2}, we need to introduce some additional notation.
Let $\ff$ be a regular thin family, $L$ be an infinite subset of $\nn$ and $k$ be a positive integer. We set
\begin{equation}
  \label{eq029(new)}
  \mathrm{Bl}_k(\ff\upharpoonright L)=\{(s_i)_{i=1}^k: (s_i)_{i=1}^k\text{ is a block sequence in }\ff\upharpoonright L\}.
\end{equation}
Let us observe that the families $\mathrm{Bl}_k(\ff\upharpoonright L)$ share the Ramsey property.
In particular, we have the following.
\begin{prop}
  \label{ramseyforblock}
  Let $\ff$ be a regular thin family, $L$ an infinite subset of $\nn$ and $k$ a positive integer.
  Then for every finite coloring of $\mathrm{Bl}_k(\ff\upharpoonright L)$ there exists an infinite subset $L'$ of $L$ such that the set $\mathrm{Bl}_k(\ff\upharpoonright L')$ is monochromatic.
\end{prop}
\begin{proof}
  By passing to an infinite subset of $L$ is necessary, we may assume that $\ff$ is very large in $L$.
  For every infinite subset $L'$ of $L$, we set
  \begin{equation}
    \label{eq040new}
    \mathrm{UBl}_k(\ff\upharpoonright L')=\Big\{\bigcup_{i=1}^k s_i:(s_i)_{i=1}^k\in\mathrm{Bl}_k(\ff\upharpoonright L')\Big\}.
  \end{equation}
  It is easy to observe that the family $\mathrm{UBl}_k(\ff\upharpoonright L)$ is a family of finite subsets of $\nn$ which is very large in $L$ and thin. Moreover, for every infinite subset $L'$ of $L$, the operator sending each $(s_i)_{i=1}^k$ from $\mathrm{Bl}_k(\ff\upharpoonright L')$ to $\bigcup_{i=1}^ks_i$ is 1-1 and onto $\mathrm{UBl}_k(\ff\upharpoonright L')$. So fixing a finite coloring on $\mathrm{Bl}_k(\ff\upharpoonright L)$ we induce a finite coloring on $\mathrm{UBl}_k(\ff\upharpoonright L)$. By the Ramsey property for thin families (see \cite{NW}, \cite{PR} or Proposition 2.6 from \cite{AKT2}) we pass to an infinite subset $L'$ of $L$ such that $\mathrm{UBl}_k(\ff\upharpoonright L')$ is monochromatic. Clearly, $\mathrm{Bl}_k(\ff\upharpoonright L')$ is monochromatic.
\end{proof}
Finally, let $\ff$ be a regular thin family and $t$ an element of $\widehat{\ff}$. We set
\begin{equation}
  \label{eq041new}
  \ff_{[t]}=\{s\in[\nn]^{<\infty}: \min s>\max t\text{ and }t\cup s\in\ff\}.
\end{equation}
It follows easily that $\ff_{[t]}$ is regular thin. We are ready to proceed to the proof of Lemma \ref{lem_cod_2}.
\begin{proof}[Proof of Lemma \ref{lem_cod_2}]
  Let $\ff$ be a regular thin family of order $\xi$.
  We define a subset $G$ of the product $[\nn]^\infty\times [\nn]^\infty\times X^{\widehat{\ff}}\times X^{\ff}$ as follows.
  We set $(M,L,(y_t)_{t\in\widehat{\ff}},(x_s)_{s\in\ff})$ in $G$ if the following are satisfied.
  \begin{enumerate}
    \item[(i)] There exists $C>0$ such that for every $k\in\nn$, every $(s_j)_{j=1}^k$ in $\mathrm{Plm}_k(\ff\upharpoonright L)$ with $s_1(1)\meg L(k)$ and every $a_1,...,a_k$ reals we have
        \begin{equation}
          \label{eq042new}
          C^{-1}\Big\|\sum_{j=1}^ka_jx_{s_j}\Big\|
          \mik\Big\|\sum_{j=1}^ka_ju_{M(j)}\Big\|
          \mik C\Big\|\sum_{j=1}^ka_jx_{s_j}\Big\|
        \end{equation}
        where $(u_n)_n$ is the standard unconditional Schauder basis of the universal space of Pe{\l}czy\'nski for the unconditional Schauder basic sequences.
    \item[(ii)] For every $s\in\ff\upharpoonright L$ we have $y_s=x_s$ and $y_\emptyset=0$.
    \item[(iii)] For every $t$ in $(\widehat{\ff}\upharpoonright L)\setminus \ff$, setting $L'=\{q\in L:q>\max t\}$, we have that for every $\ee>0$ there exists $n_0\in\nn$ such that for every $n\meg n_0$ and every block sequence $(t'_j)_{j=1}^n$ in $\ff_{[t]}\upharpoonright L$ with $\min t'_1\meg L'(n)$ we have
        \begin{equation}
          \label{eq043new}
          \Big\|\frac{1}{n}\sum_{j=1}^n y_{t\cup t'_j}-y_t\Big\|\mik\ee.
        \end{equation}
  \end{enumerate}
  It is easy to check that $G$ is a Borel subset of $[\nn]^\infty\times[\nn]^\infty\times X^{\widehat{\ff}_k}\times X^{\ff_k}$.
  We have the following claim.

  \noindent\textbf{Claim 1:} Let $(M,L,(y_t)_{t\in\widehat{\ff}},(x_s)_{s\in\ff})$ in $G$. Then every $\ff$-spreading model generated by an $\ff$-subsequence of $(x_s)_{s\in\ff\upharpoonright L}$ belongs to $SM_\xi^w(X)$ and it is equivalent to $(u_n)_{n\in M}$.
  \begin{proof}[Proof of Claim 1]
    First observe that the second part of the conclusion of Claim 1 is immediate by property (i) above.
    Let $\widehat{\varphi}:\widehat{\ff}\upharpoonright L\to X$ such that $\widehat{\varphi}(t)=y_t$ for all
    $t\in\widehat{\ff}\upharpoonright L$. We need to show that $\widehat{\varphi}$ is continuous, where $X$ is considered with the weak topology.
    By Lemma \ref{pushing_cont}, it suffices to show that for every $t\in \widehat{\ff}\upharpoonright L$ and every $(s_n)_n$ in $\ff\upharpoonright L$ convergent to $t$, we have that $y_{s_n}\stackrel{w}{\to}y_t$.
    Assume to the contrary that there exist $t\in \widehat{\ff}\upharpoonright L$, a sequence $(s_n)_n$ in $\ff\upharpoonright L$ convergent to $t$, an element $x^*$ in $X^*$ of norm $1$ and some $\ee>0$ such that $x^*(y_{s_n}-y_t)\meg2\ee$ for all $n$. Passing to a subsequence of $(s_n)_n$ if necessary, we may assume that each $s_n$ end-extends $t$ and $(s_n\setminus t)_n$ is a block sequence. We set $t_n=s_n\setminus t$ for all $n$. Then for every $n\in\nn$ we get that
    \begin{equation}
      \label{eq044new}
      \Big\|\frac{1}{n}\sum_{j=1}^ny_{s_{n+j}}-y_t\Big\|\meg x^*\Big(\frac{1}{n}\sum_{j=1}^ny_{s_{n+j}}-y_t\Big)\meg2\ee,
    \end{equation}
    which contradicts (iii) above. Hence $(x_s)_{s\in\ff\upharpoonright L}$ is subordinated. Moreover, by (ii), it follows that $\widehat{\varphi}(\emptyset)=0$ and therefore $(x_s)_{s\in\ff\upharpoonright L}$ is weakly null. The proof of the claim is complete.
  \end{proof}
  The converse of Claim 1 holds as well.

  \noindent\textbf{Claim 2:} For every element $(e_n)_n$ of
  $SM_\xi^w(X)$ there exists $(M,L,(y_t)_{t\in\widehat{\ff}},(x_s)_{s\in\ff})$ in $G$ such that the $\ff$-subsequence $(x_s)_{s\in\ff\upharpoonright L}$ generates $(e_n)_n$ as an $\ff$-spreading model.
  \begin{proof}[Proof of Claim 2]
    Fix $(e_n)_n$ in $SM_\xi^w(X)$. By the universality property of Pe{\l}czy\'nski's space, there exists an infinite subset $M$ of $\nn$ such that the sequences $(e_n)_n$ and $(u_n)_{n\in M}$ are equivalent. By Corollary \ref{generating_subord_sm}, there exist an infinite subset $P$ of $\nn$ and an $\ff$-sequence $(x_s)_{s\in\ff}$ such that the
    $\ff$-subsequence $(x_s)_{s\in\ff\upharpoonright P}$ is subordinated, weakly null and generates $(e_n)_n$ as an $\ff$-spreading model. For every $t\in\widehat{\ff}\upharpoonright P$ we set $y_t=\widehat{\varphi}(t)$ and we pick an arbitrary $y_t$ for every
    $t\in \widehat{\ff}\setminus(\widehat{\ff}\upharpoonright P)$.
    Clearly for every infinite subset $P'$ of $P$,
    $(M,P',(y_t)_{t\in\widehat{\ff}},(x_s)_{s\in\ff})$ satisfies (i)
    and (ii). It suffices to choose an infinite subset $L$ of $P$ such that $(M,L,(y_t)_{t\in\widehat{\ff}},(x_s)_{s\in\ff})$  satisfies (iii).
    First, let us observe that the following property holds true.
    \begin{enumerate}
      \item[($\mathcal{P}$)] For every infinite subset $P'$ of $P$ and every $t$ in $(\widehat{\ff}\upharpoonright M)\setminus \ff$,
      there exists an infinite subset $P''$ of $P'$ such that $\min P''>\max t$ and for every $n\in\nn$ and every
      block sequences $(t_j)_{j=1}^n,(t'_j)_{j=1}^n$ in $\ff_{[t]}\upharpoonright P''$ with
      $\min t_1,\min t'_1\meg P''(n)$ we have that
      \[\Bigg|\Big\|\frac{1}{n}\sum_{j=1}^ny_{t\cup t_j}-y_t\Big\|-\Big\|\frac{1}{n}\sum_{j=1}^ny_{t\cup t'_j}-y_t\Big\|\Bigg|\mik\frac{1}{n}.\]
    \end{enumerate}
  Indeed, let $P'$ and $t$ as above. We set $Q_0=P'$ and we inductively construct a decreasing
  sequence $(Q_n)_n$ of infinite subsets of $Q_0$ such that for every $n\in\nn$ we have that
  \begin{equation}
    \label{eq045new}
    \Bigg|\Big\|\frac{1}{n}\sum_{j=1}^ny_{t\cup t_j}-y_t\Big\|-\Big\|\frac{1}{n}\sum_{j=1}^ny_{t\cup t'_j}-y_t\Big\|\Bigg|\mik\frac{1}{n},
  \end{equation}
  for every choice of block sequences $(t_j)_{j=1}^n,(t'_j)_{j=1}^n$ in $\ff_{[t]}\upharpoonright Q_n$.
  Assume that for some $n\in\nn$ the sets $Q_0,...,Q_{n-1}$ have been chosen.
  Since $\widehat{\varphi}$ is continuous and $\widehat{\ff}\upharpoonright Q_{n-1}$ is compact, the set $\{y_s:s\in\widehat{\ff}\upharpoonright Q_{n-1}\}$ is weakly compact and therefore bounded. Let $C>0$ such that $\|y_s\|\mik C$ for all $s\in\widehat{\ff}\upharpoonright Q_{n-1}$. Let $(A_i)_{i=1}^{i_0}$ be a partition of $[0,C]$ into sets of diameter at most $1/n$. We define a finite coloring on $\mathrm{Bl}_n(\ff_{[t]}\upharpoonright Q_{n-1})$ as follows. We assign to a block sequence $(t_j)_{j=1}^n$ in $\ff_{[t]}\upharpoonright Q_{n-1}$ the color $i\in\{1,...,i_0\}$ if
  \begin{equation}
     \label{eq046new}
     \Big\|\frac{1}{n}\sum_{j=1}^{n}y_{t\cup t_j}-y_{t}\Big\|\in A_i.
  \end{equation}
  Applying Proposition \ref{ramseyforblock} we obtain an infinite subset $Q_n$ of $Q_{n-1}$ such that the set
  $\mathrm{Bl}_n(\ff_{[t]}\upharpoonright Q_n)$ is monochromatic. It is easy to check that $Q_n$ is as desired
  and the inductive step is complete. Pick an infinite subset $P''$ of $P'$ such that $P''(n)\in Q_n$ for all $n$.

  Set $L_0=P$ and inductively construct a decreasing sequence $(L_q)_q$ of infinite subsets of $P$ and a strictly increasing sequence $(l_q)_q$ in $P$ satisfying the following for every $q\in\nn$.
    \begin{enumerate}
      \item[(a)] $l_q=\min L_q$.
      \item[(b)] For every $t$ subset of $\{l_p:1\mik p\mik q-1\}$ belonging to $\widehat{\ff}\setminus \ff$, every $n\in\nn$ and every block sequences $(t_j)_{j=1}^n, (t'_j)_{j=1}^n$ in $\ff_l\upharpoonright L_q$ with $\min t_1,\min t'_1\meg L_q(n)$ we have that
      \begin{equation}
        \label{eq047new}
        \Bigg|\Big\|\frac{1}{n}\sum_{j=1}^ny_{t\cup t_j}-y_t\Big\|-\Big\|\frac{1}{n}\sum_{j=1}^ny_{t\cup t'_j}-y_t\Big\|\Bigg|\mik\frac{1}{n}.
      \end{equation}
    \end{enumerate}
    The inductive step of the construction is carried out as follows. Assume that for some $q\in\nn$ the
    sets $L_0,...,L_{q-1}$ are constructed. Let $\{t^r\}_{r=1}^N$ be an enumeration of
    the set $\big\{t\in\widehat{\ff}\setminus\ff:t\subseteq\{l_p:1\mik p\mik q-1\}\big\}$.
    Applying property $(\mathcal{P})$, we construct a decreasing
    sequence $(Q_r)_{r=1}^N$ of infinite subsets of $L_{q-1}\setminus\{l_{q-1}\}$ such that for every $n\in\nn$ and every block sequences $(t_j)_{j=1}^n, (t'_j)_{j=1}^n$ in $\ff_{[t]}\upharpoonright Q_r$ with $\min t_1,\min t'_1\meg Q_r(n)$ we have that
      \begin{equation}
        \label{eq048new}
        \Bigg|\Big\|\frac{1}{n}\sum_{j=1}^ny_{t^r\cup t_j}-y_{t^r}\Big\|-\Big\|\frac{1}{n}\sum_{j=1}^ny_{t^r\cup t'_j}-y_{t^r}\Big\|\Bigg|\mik\frac{1}{n}.
      \end{equation}
      Setting $L_q=Q_N$, the inductive step of the construction is complete.

      We set $L=\{l_q:q\in\nn\}$. Then
      for every $t$ in $(\widehat{\ff}\upharpoonright L)\setminus \ff$,
      every $n\in\nn$ and every block sequences $(t_j)_{j=1}^n, (t'_j)_{j=1}^n$ in $\ff_{[t]}\upharpoonright L'$, where $L'=\{q\in L:q\meg\max t\}$, with $\min t_1,\min t'_1\meg L'(n)$ it follows that
      \begin{equation}
        \label{eq049new}
        \Bigg|\Big\|\frac{1}{n}\sum_{j=1}^ny_{t\cup t_j}-y_t\Big\|-\Big\|\frac{1}{n}\sum_{j=1}^ny_{t\cup t'_j}-y_t\Big\|\Bigg|\mik\frac{1}{n}.
      \end{equation}
      In order to check that property (iii) is satisfied we fix some $t$ from
      $(\widehat{\ff}\upharpoonright L)\setminus \ff$.
      Let $L'=\{q\in L:q>\max t\}$. Let $(t_n)_n$ be a block sequence in $\ff_{[t]}\upharpoonright L'$. By the lemma's assumptions, the sequence $(y_{t\cup t_n}-y_t)_n$ admits no spreading model equivalent to the standard basis of $\ell^1$. Moreover, by the continuity of the map $\widehat{\varphi}$,  $(y_{t\cup t_n}-y_t)_n$ is weakly null. Hence, by a well known dichotomy of H.P. Rosenthal concerning Ces\`aro summability and $\ell^1$ spreading models, it follows that there exists a subsequence $(y_{t\cup t_{m_n}}-y_t)_n$ of $(y_{t\cup t_n}-y_t)_n$ which is Ces\`aro summable to zero. Hence
      \begin{equation}
        \label{eq050new}
        \begin{split}
          \lim_{n\to\infty}\Big\|\frac{1}{n}\sum_{j=1}^n & y_{t\cup t_{m_{n+j}}}-y_{t}\Big\| \\
          &\mik 2\lim_{n\to\infty}\Big\|\frac{1}{2n}\sum_{j=1}^{2n}y_{t\cup t_{m_{j}}}-y_{t}\Big\|
          -\lim_{n\to\infty}\Big\|\frac{1}{n}\sum_{j=1}^ny_{t\cup t_{m_{j}}}-y_{t}\Big\|=0.
        \end{split}
      \end{equation}
      Clearly, (iii) follows by \eqref{eq049new} and \eqref{eq050new}. The proof of the claim is complete.
  \end{proof}
  Let $A$ to be the projection of $G$ to the first coordinate, that is
  \begin{equation}
          \label{eq051new}
          \begin{split}
            A=\{M\in[\nn]^\infty:\;&\text{there exists}\;(L,(y_t)_{t\in\widehat{\ff}},(x_s)_{s\in\ff})\in[\nn]^\infty\times X^{\widehat{\ff}}\times X^{\ff}\;\\
          &\text{such that}(M,L,(y_t)_{t\in\widehat{\ff}},(x_s)_{s\in\ff})\in G\;\}.
          \end{split}
        \end{equation}
  Since $G$ is Borel, $A$ is analytic. It remains to check that $A$ satisfies (i) and (ii) of the lemma. Indeed, let $(e_n)_n$ in $SM_\xi^w(X)$. By Claim 2, there exists $(M,L,(y_t)_{t\in\widehat{\ff}},(x_s)_{s\in\ff})$ in $G$ such that the $\ff$-subsequence $(x_s)_{s\in\ff\upharpoonright L}$ generates $(e_n)_n$ as an $\ff$-spreading model. By the definition of $A$, $M$ belongs to $A$ and by property (i) the sequences $(e_n)_n$ and $(u_n)_{n\in M}$ are equivalent.
  Conversely, let $M\in A$. By the definition of $A$, there exist an infinite subset $L$ of $\nn$, an $\ff$-sequence $(x_s)_{s\in\ff}$ in $X$ and a family $(y_t)_{t\in\widehat{\ff}}$ of elements in $X$ such that
  $(M, L,(y_t)_{t\in\widehat{\ff}},(x_s)_{s\in\ff})$ belongs to $G$.
  We pass to an infinite subset $L'$ of $L$ such that the $\ff$-subsequence $(x_s)_{s\in\ff\upharpoonright L'}$ generates an $\ff$-spreading model $(e_n)_n$. By Claim 1 the sequences $(e_n)_n$ and $(u_n)_{n\in M}$ are equivalent
  and that $(e_n)_n$ belongs to $SM_\xi^w(X)$. The proof is complete.
\end{proof}

\end{document}